\documentclass[12pt,leqno,a4paper]{amsart}
\title{Symmetric products of a semistable degeneration of surfaces}
\author{Yasunari Nagai}

\address{
Department of Mathematics, Faculty of Science and Engineering,
Waseda University,
3-4-1 Ohkubo, Shinjuku, Tokyo 169-8555, Japan}
\email{nagai.y@waseda.jp}

\usepackage{amsthm}
\usepackage{amssymb}
\usepackage{amsxtra}
\usepackage[mathscr]{eucal}
\usepackage{enumerate}
\usepackage[all]{xy}
\usepackage[initials,alphabetic,lite]{amsrefs}

\usepackage{times,mathptmx,bm}

\theoremstyle{plain}
\newtheorem{theorem}[subsection]{Theorem}
\newtheorem{lemma}[subsection]{Lemma}
\newtheorem{proposition}[subsection]{Proposition}

\newtheorem*{theorem*}{Theorem}
\newtheorem*{corollary*}{Corollary}

\theoremstyle{definition}

\theoremstyle{remark}
\newtheorem*{claim}{Claim}
\newtheorem{remark}[subsubsection]{Remark}

\DeclareSymbolFont{cmletters}{OML}{cmm}{m}{it}
\DeclareSymbolFont{cmsymbols}{OMS}{cmsy}{m}{n}
\DeclareSymbolFont{cmlargesymbols}{OMX}{cmex}{m}{n}
\DeclareMathSymbol{\myjmath}{\mathord}{cmletters}{"7C}
\DeclareMathSymbol{\myamalg}{\mathbin}{cmsymbols}{"71}
\DeclareMathSymbol{\mycoprod}{\mathop}{cmlargesymbols}{"60}
\let\jmath\myjmath
\let\amalg\myamalg
\let\coprod\mycoprod

\def\lto{\longrightarrow}

\DeclareMathOperator{\codim}{codim}

\DeclareMathOperator{\Hilb}{Hilb}
\DeclareMathOperator{\Sym}{Sym}

\DeclareMathOperator{\Spec}{Spec}

\DeclareMathOperator{\adj}{adj}
\DeclareMathOperator{\id}{id}

\DeclareMathOperator{\Hom}{Hom}

\DeclareMathOperator{\Ker}{Ker}

\DeclareMathOperator{\Stab}{Stab}
\DeclareMathOperator{\Conj}{Conj}
\DeclareMathOperator{\age}{age}
\let\div\relax
\DeclareMathOperator{\div}{div}

\def\git{/\!\!/}

\newcommand\numberthis{\addtocounter{equation}{1}\tag{\theequation}}

\makeatletter
\def\@seccntformat#1{%
  \protect\textup{\protect\@secnumfont
    \ifnum\pdfstrcmp{subsection}{#1}=0 \bfseries\fi% subsection # in \bfseries
    \csname the#1\endcsname
    \protect\@secnumpunct
  }%
}
\makeatother

\BibSpec{article}{%
    +{}  {\PrintAuthors}                {author}
    +{,} { \textit}                     {title}
    +{.} { }                            {part}
    +{:} { \textit}                     {subtitle}
    +{,} { \PrintContributions}         {contribution}
    +{.} { \PrintPartials}              {partial}
    +{,} { }                            {journal}
    +{}  { \textbf}                     {volume}
    +{}  { \PrintDatePV}                {date}
    +{,} { \issuetext}                  {number}
    +{,} { \eprintpages}                {pages}
    +{,} { }                            {status}
    +{,} { \url}                        {url}    % <---- ADDED
    +{,} { \PrintDOI}                   {doi}
    +{,} { available at \eprint}        {eprint}
    +{}  { \parenthesize}               {language}
    +{}  { \PrintTranslation}           {translation}
    +{;} { \PrintReprint}               {reprint}
    +{.} { }                            {note}
    +{.} {}                             {transition}
    +{}  {\SentenceSpace \PrintReviews} {review}
}

\begin{document}

\baselineskip 17pt
\parskip 5pt

\maketitle
\vspace{-1cm}

\begin{abstract}
We explicitly construct a $V$-normal crossing Gorenstein canonical model of
the relative symmetric products of a local semistable degeneration of surfaces without a triple point
by means of toric geometry.
Using this model, we calculate the stringy $E$-polynomial of the relative symmetric product.
We also construct a minimal model of degeneration of Hilbert schemes explicitly.
\end{abstract}

%%%%%%%%%%%
\section*{Introduction}

Let $S$ be a smooth (quasi-)projective algebraic surface. A theorem of Fogarty \cite{F} says that
the Hilbert scheme $\Hilb ^n(S)$ of 0-dimensional subschemes on $S$ of length $n$ is a smooth (quasi-)projective
algebraic variety of dimension $2n$. This construction gives a very nice and interesting
way to produce higher dimensional
algebraic varieties. For example, if $S$ is a K3 surface (resp. an abelian surface),
then $\Hilb ^n(S)$ (resp. the albanese
fiber of $\Hilb ^{n+1}(S)$) gives an example of higher dimensional
irreducible symplectic compact K\"ahler manifold \cite{Be}.
Besides the holomorphic symplectic geometry, the Hilbert scheme of points on a surface is related to many branches of
mathematics, such as differential geometry, singularity theory, and representation theory.

As Hilbert schemes behave nicely in family, it is quite natural to think
of the relative Hilbert scheme $\Hilb ^n(\mathcal S/B)$
for a flat family of surfaces $\pi: \mathcal S\to B$.
If the family $\pi$ degenerates at some point $b\in B$, one naturally expects
that the family of Hilbert schemes $\Hilb ^n(\mathcal S/B)\to B$ also degenerates at $b$.
In this setting, one of the fundamental questions is to ask how much singular the induced degeneration
of Hilbert schemes is. Of course, it will depend on
the singularity of the degeneration of the original family. To get a modest degeneration of Hilbert schemes,
it is natural to assume that the family of surfaces $\pi :\mathcal S\to B$ is {\it semistable}.
In such a situation, another natural question is to find a \emph{good} birational model of
a degeneration of Hilbert schemes that is semistable (or very near to semistable) and
minimal \cite{N,GHH}, and to understand the behavior of the family.

However, even if the family $\pi:\mathcal S\to B$ is a semistable degeneration,
and hence $\mathcal S$ is smooth, it seems difficult, at least to the author,
to study the relative Hilbert scheme $\Hilb ^n(\mathcal S/B)$ directly by the
ring theoretic approach as in \cite{F} or \cite{Br},
in contrast to Ran's work \cite{R} on the case of semistable degeneration
of curves in this direction.
In fact, our relative Hilbert scheme
can be seen as a closed subscheme $\Hilb ^n(\mathcal S/B)\subset \Hilb ^n(\mathcal S)$,
while $\Hilb ^n(\mathcal S)$ can be very singular
for $\dim \mathcal S\geqslant 3$ and $n$ large.

In this article, we will focus on the {\it relative symmetric product} $\Sym ^n(\mathcal S/B)$
rather than the relative Hilbert scheme, and study its singularity and birational geometry.
For that purpose, we start from a local model of semistable degeneration of surfaces
$\mathcal S\to B$ and describe the symmetric product $\Sym ^n(\mathcal S/B)$ as a quotient
of certain affine toric variety by an action of the symmetric group (\S 1).
This description leads us to a Gorenstein canonical model with only quotient singularities
for a degeneration without a triple point (\S 2, Theorem \ref{gorenstein canonical model}).
The Gorenstein canonical model is obtained as a quotient by a natural action
of the symmetric group of the total space of a rank two toric vector bundle over a toric
variety associated with the Coxeter complex of a root system of type A (Proposition \ref{description coxeter}).
It is noteworthy that such a toric variety was studied by several authors
from an interest of combinatorics and representation of symmetric groups \cite{P,D,S}.
This toric-quotient description also enables us to calculate
the stringy $E$-polynomial of the Gorenstein canonical model, which
encodes cohomological information of the degeneration (\S 3, Theorem \ref{formula for stringy E-polynomial},
Proposition \ref{stringy E-poly for small n}).
In the last section,
we discuss an explicit construction of a $\mathbb Q$-factorial terminal minimal model of
the relative symmetric product (Theorem \ref{minimal model}),
which turns out to be a $V$-normal crossing degeneration of Hilbert schemes of general fibers.
This gives a good birational model of a degeneration of Hilbert schemes in the case where
the singular fiber of the original semistable degeneration of surfaces has no triple point.
We also relate our minimal model to the relative Hilbert scheme explicitly in the case
where the length $n=2$.

\paragraph{\bf Acknowledgement}
The author would like to thank M. Brion, M. Lehn, and T. Yasuda for their interest
and helpful comments. He also thanks the anonymous referee for giving him
valuable comments.
This work is partially supported by
JSPS Grants-in-aid for young scientists (B) 26800025.

\section{Local description of symmetric product}

\subsection{}\label{setting}
Let $S_2=S_3=\mathbb C^3$ and define $p _d:S_d\to B=\mathbb C$ by
\[
p _d(x_1,x_2,x_3)=
\begin{cases}
x_1x_2 & \mbox{if } d=2\, ,\\
x_1x_2x_3 & \mbox{if } d=3\, .
\end{cases}
\]
The origin of $S_2$ is the local model at the general point of the singular locus
of the singular fiber of semistable degeneration of surfaces,
while the origin of $S_3$ is the maximally degenerate point.
Let us denote the $n$-fold self-products of $S_d$ relative to $p _d$ by
\[
\widetilde X^{(n)}_d=
(S_d/B)^n = S_d \times _B S_d \times _B \cdots \times _B S_d.
\]
and let $\tilde \pi ^{(n)}_d: \widetilde X^{(n)}_d\to B$ be the natural morphism.
Then, the symmetric group $\mathfrak S_n$ acts on $\widetilde X^{(n)}_d$
and $\tilde \pi ^{(n)}_d$ is $\mathfrak S_n$-invariant. The quotient variety
\[
\pi ^{(n)}_d: X^{(n)}_d=\widetilde X^{(n)}_d/\mathfrak S_n \to B
\]
is the $n$-th \emph{relative symmetric product} of $p_d:S_d\to B$.

The fiber $(\pi ^{(n)}_d)^{-1}(b)$ for $b\in B,\, b\neq 0$ is just
$\Sym ^n(\mathbb C^*\times \mathbb C)$ for $d=2$ and $\Sym ^n((\mathbb C^*)^2)$ for $d=3$.
It is also easy to see the combinatorics of the fiber over the origin.
It has $\binom{n+d-1}{d}$ components each of which corresponds to a partition of $n$,
\[
\mathbf a=(a_i\; |\; 1\leqslant i\leqslant d,\; \sum a_i = n).
\]
The component $X^{(n)}_{\mathbf a}$ is an image of a birational morphism
\[
\Sym ^{\mathbf a}(\mathbb C^2) = \prod _{i=1}^d \Sym ^{a_i}(\mathbb C^2)
\to X^{(n)}_{\mathbf a}.
\]
The morphism has finite fibers over the double curves except
for the case $a_i=n$ for some $i$. It implies that $X^{(n)}_{\mathbf a}$ is
\emph{non-normal} for general $\mathbf a$. The geometry of the intersection of
these components seems difficult to describe directly.

In contrast, it is easy to describe the product variety
$\widetilde X^{(n)}_d$ by affine equations.
Let $(z_{11},z_{12},z_{13},\cdots, z_{n1},z_{n2},z_{n3})$ be
the coordinate of $(\mathbb C^3)^n=\mathbb C^{3n}$. Then, $\widetilde X^{(n)}_d$
is nothing but the complete intersection
\begin{equation}\label{generalized nc}
z_{11}\cdots z_{1d}=z_{21}\cdots z_{2d}=\cdots =z_{n1}\cdots z_{nd}\quad (d=2,3)
\end{equation}
of dimension $2n+1$ and the projection $\tilde \pi ^{(n)}_d: \widetilde X^{(n)}_d\to B=\mathbb C$ is
the function defined by the value of these monomials. If $d=2$, the variety
split into a product of the closed subvariety $\widetilde X^{(n) \prime}_2\subset
\mathbb C^{2n}$ defined by the same equation \eqref{generalized nc} and $\mathbb C^n$.

\begin{proposition}\label{Q-Gorenstein}
$X^{(n)}_d$ is $\mathbb Q$-Gorenstein, \emph{i.e.}, the canonical divisor of $X^{(n)}_d$ is
$\mathbb Q$-Cartier.
\end{proposition}

\begin{proof}
The locus $F$ of points with non-trivial stabilizers with respect to the action of $\mathfrak S_n$ is
the union of linear subspaces defined by $z_{i1}=z_{j1},\, z_{i2}=z_{j2}$, and $z_{i3}=z_{j3}$
for $i\neq j$. Therefore, the codimension of $F$ inside $\widetilde X^{(n)}_d$ is
two. This implies that the quotient map $\widetilde X^{(n)}_d\to X^{(n)}_d$ is \'etale in codimension 1
and the proposition follows from Proposition 5.20 of \cite{KM}.
\end{proof}

\subsection{}
For later use, we give a description of $\widetilde X^{(n)}_d$ as a
toric variety. Let us first consider the case in which $d=2$.
Let $M=\mathbb Z^{n+1}$ and $N=\Hom _{\mathbb Z}(M,\mathbb Z)$ its dual.
We denote by $[a_1\, a_2\, \cdots a_l]_{lr}$ the sequence $a_1, a_2, \cdots, a_l$ recurred
$r$ times. For example,
\[
[0\, 0\,1\, 1]_{12} = ( 0, 0, 1,1, 0, 0, 1,1, 0, 0, 1,1).
\]
Using this notation, we define $C^{(n)}_2$ as the $(n+1)\times 2^n$ matrix
whose rows are given by
\[
\begin{aligned}
{[} &1\, 0] _{2^n}\\
[ &0\, 1]_{2^n}\\
[ &0\, 0\, 1\, 1]_{2^n}\\
&\vdots\\
[ & 0^{2^{n-1}}\,  1^{2^{n-1}}]_{2^n}\; ,
\end{aligned}
\]
where `` $0^r$ '' means $0$ repeated $r$-times and the same for `` $1^r$ ''.
For example,
\[
C^{(3)}_2=
\begin{pmatrix}
1 & 0 & 1 & 0 & 1 & 0 & 1 & 0\\
0 & 1 & 0 & 1 & 0 & 1 & 0 & 1\\
0 & 0 & 1 & 1 & 0 & 0 & 1 & 1\\
0 & 0 & 0 & 0 & 1 & 1 & 1 & 1
\end{pmatrix} .
\]
Under the standard identification $N\cong \mathbb Z^{n+1}$, we define
$\sigma ^{(n)}_2\subset N\otimes \mathbb R$ to be a rational polyhedral convex cone generated by
the column vectors in $C^{(n)}_2$. The cone $\sigma ^{(n)}_2$ is a non-simplicial cone
of maximal dimension for $n\geqslant 2$.

\begin{proposition}\label{cone sigma(n)}
The affine variety $\widetilde X^{(n)\prime}_2\subset \mathbb C^{2n}$ defined above is the affine toric variety
$X(\sigma ^{(n)}_2)=\Spec \mathbb C[\sigma ^{(n)\vee}_2\cap M]$.
\end{proposition}

This proposition is implicitly given in \cite{W}, \S4. Here we give a proof of slightly different flavor.

\begin{lemma}\label{fiber product of cones}
Let $\sigma _1,\sigma _2,\bar \sigma$ be strictly convex rational polyhedral cones on
lattices $N_1, N_2, \bar N$, respectively. Assume that we have surjective homomorphisms
$h_1:N_1\to \bar N$ and $h_2:N_2\to \bar N$ such that
$\bar \sigma =h_{1,\mathbb R}(\sigma _1)=h_{2,\mathbb R}(\sigma _2)$.
Let $\pi _i:X(\sigma _i)\to X(\bar \sigma)\; (i=1,2)$ be the corresponding toric morphisms of
affine toric varieties. Then, the fiber product $X(\sigma _1)\times _{X(\bar \sigma)}X(\sigma _2)$
is an affine toric variety corresponding to the cone
\[
\sigma _1\times _{\bar N} \sigma _2
=\{(v_1,v_2)\in (N_1\oplus N_2)_\mathbb R\, |\,v_i\in \sigma _i\,(i=1,2),\; h_1(v_1)=h_2(v_2)\}
\]
on the lattice $N_1\times _{\bar N}N_2=\{(v_1,v_2)\in N_1\oplus N_2\,|\, h_1(v_1)=h_2(v_2)\}$.
\end{lemma}

\begin{proof}
Let $M_1,M_2, \bar M$ be dual lattices of $N_1,N_2,\bar N$, respectively.
Since $h_i\; (i=1,2)$ is surjective, $\bar M$ is a direct summand of $M_i$.
The tensor product
\[
\mathbb C[\sigma _1^{\vee}\cap M_1] \underset{\mathbb C[\bar \sigma^{\vee} \cap \bar M]}{\otimes}
\mathbb C[\sigma _2^{\vee}\cap M_2]
\]
has a basis consisting of monomials $(m_1,m_2)\, (m_i\in M_i)$ subject to a relation
\[
(m_1,m_2)=(m_1',m_2') \mbox{ if and only if } m_1-m_1'=m_2'-m_2\in \bar M.
\]
This implies that the tensor ring is the monoid ring corresponding to
a cone $C$ that is the image of the product cone
$\sigma _1^{\vee}\times \sigma _2^{\vee}$ under the surjective homomorphism to the fiber co-product
\[
M_1\oplus M_2 \to M_1+_{\bar M}M_2=(M_1\oplus M_2)/\bar M.
\]
In particular, $C$ is spanned by the vectors $(m_1,m_2)\in M_1+_{\bar M}M_2$ with
$m_i$ a generator of a ray of $\sigma _i$ for each $i=1,2$.
Passing to the dual, the dual cone $C^{\vee}$ is cut out by positive half-planes
defined by $(m_1,m_2)$ as above on the fiber product of lattices $N_1\times _{\bar N}N_2$.
This immediately implies that $C^{\vee}$ is nothing but the fiber product of cones
$\sigma _1\times _{\bar N} \sigma _2$.
\end{proof}

\begin{proof}[Proof of Proposition \ref{cone sigma(n)}]
 Let $C\to B=\mathbb C$ be a family of curves defined by $(x_1,x_2)\mapsto x_1x_2$.
 This is a toric morphism corresponding to a surjective homomorphism of lattices
 \[
 (1\; 1):N_C=\mathbb Z^2\to N_B=\mathbb Z
 \]
 that is compatibile with the cones $\sigma _C=\mathbb R_{\geqslant 0} \, d_1
 +\mathbb R_{\geqslant 0} \, d_2$ and
 $\sigma _B=\mathbb R_{\geqslant 0}$,
 where $d_1=\begin{pmatrix} 1 \\ 0 \end{pmatrix}$ and $d_2=\begin{pmatrix} 0 \\ 1 \end{pmatrix}$.
 Noting that $\widetilde X^{(n)\prime}_2$ is an $n$-fold fiber product $(C/B)^n$,
 the cone $\sigma ^{(n)}_2$ is nothing but the fiber product of cones
 \[
 (\sigma _C/N_B)^n = \sigma _C \times _{N_B} \cdots \times _{N_B} \sigma _C
 \]
 on the lattice $(N_C/N_B)^n$
 by the lemma above. If we take a basis $(N_C/N_B)^n\cong \mathbb Z^{n+1}$ as
 \[
 \begin{aligned}
 e_0 & = (d_1,d_1,\dots, d_1),\quad e_1 = (d_2,d_1,\dots, d_1) \\
 e_j & = (0, \dots, \underset{\scriptsize \substack{\wedge \\ (j+1)}}{d_2-d_1},0,\dots, 0)
 \end{aligned}
 \]
 one immediately sees that the cone $(\sigma _C/N_B)^n$ is generated by the column vectors
 of the matrix $C^{(n)}_2$ defined above.
\end{proof}

\subsection{}\label{symmetric group action}
Let $\widetilde M=\mathbb Z^2 \otimes \mathbb Z^n$. The symmetric group $\mathfrak S_n$ acts on $\widetilde M$ by
the permutation representation on the second factor. Define a $(n+1)\times 2n$ integral matrix $P^{(n)}_2$ by
\begin{equation}\label{P n 2}
P^{(n)}_2=(e_0\quad e_1\quad e_0+e_1-e_2\quad e_2\quad \cdots \quad e_0+e_1-e_n\quad e_n),
\end{equation}
where $e_0,\cdots, e_n$ is the standard basis of $M=\mathbb Z^{n+1}$.
The cone $\sigma ^{(n)\vee}_2$
is nothing but the image under the surjective linear map
$P^{(n)}_2:\widetilde M\otimes \mathbb R \to M\otimes \mathbb R$ of
the cone spanned by the standard simplex in $\widetilde M$.
Therefore, we have an induced action of $\mathfrak S_n$ on $M$ and its dual $N$.
More precisely, $\mathfrak S_n$ acts on $M$ by permuting $n$ pairs of vectors
\[
(e_0\quad e_1 \quad| \quad e_0+e_1-e_2\quad e_2\quad |\quad \cdots \quad
|\quad e_0+e_1-e_n\quad e_n),
\]
so that the action of $\mathfrak S_n$ on $N$ is represented by matrices
\begin{equation}\label{Sn action on N}
(1\; 2)=
\left(
\begin{array}{ccc|c}
1 & 1 & -1 & \\
0 & 0 & 1 & \\
0 & 1 & 0 & \\ \hline
& & & I_{n-2}
\end{array}
\right)\mbox{\; and \;}
(k\;\; k+1)=
\left(
\begin{array}{c|cc|c}
I_k & & &\\ \hline
& 0 & 1 &\\
& 1 & 0 & \\ \hline
& & & I_{n-k-1}
\end{array}
\right)
\end{equation}
for $k>1$.
The cone $\sigma ^{(n)}_2$ and its dual $\sigma ^{(n)\vee}_2$ are invariant under
the action of $\mathfrak S_n$. Let $\sigma$ be a cone in
$(N\oplus \mathbb Z^n)\otimes  \mathbb R$ spanned by
$\sigma ^{(n)}_2$ and the standard basis of $\mathbb Z^n$. Then,
the associated affine toric variety $X(\sigma)$ is nothing but $\widetilde X^{(n)}_2$.
The action of $\mathfrak S_n$ on $\widetilde X^{(n)}_2$ coincides with
the action induced by the diagonal $\mathfrak S_n$-action on $N$ and $\mathbb Z^n$.

\subsection{}
We also have a similar description of $\widetilde X^{(n)}_3$ as a toric variety.
Let $M=\mathbb Z^{2n+1}$ and $N=\Hom _{\mathbb Z}(M,\mathbb Z)\cong \mathbb Z^{2n+1}$.
Let $C^{(n)}_3$ be the $(2n+1)\times 3^n$ matrix whose rows are given by
\[
\begin{aligned}
{}[ & 1\, 0\, 0]_{3^n}\\
[ & 0\, 1\, 0]_{3^n}\\
[ & 0\, 0\, 1]_{3^n}\\
[ & 0\, 0\, 0\, 1\, 1\, 1\, 0\, 0\, 0]_{3^n}\\
[ & 0\, 0\, 0\, 0\, 0\, 0\, 1\, 1\, 1]_{3^n}\\
& \vdots \\
[ & 0^{3^{n-1}}\, 1^{3^{n-1}}\, 0^{3^{n-1}} ]_{3^n}\\
[ & 0^{3^{n-1}}\, 0^{3^{n-1}}\, 1^{3^{n-1}} ]_{3^n}\; .
\end{aligned}
\]
and $\sigma ^{(n)}_3\subset N\otimes \mathbb R$ the cone generated
by the column vectors of $C^{(n)}_3$. Then, the associated toric
variety $X(\sigma ^{(n)}_3)=\Spec \mathbb C[\sigma ^{(n)\vee}_3\cap M]$ is
nothing but $\widetilde X^{(n)}_3$. The proof of this claim is completely parallel
to the case of $\widetilde X^{(n)\prime}_2$; it is a direct consequence of Lemma \ref{fiber product of cones}.

\section{Gorenstein canonical orbifold model}

In this section, we construct a Gorenstein canonical model of
$X^{(n)}_2$ with only quotient singularities. From now on, we concentrate on the case
$d=2$. We suppress the subscript and \underline{write $\widetilde X^{(n)}$, $\widetilde X^{(n)\prime}$ and $X^{(n)}$
instead of$\widetilde X^{(n)}_2$, $\widetilde X^{(n)\prime}_2$ and $X^{(n)}_2$},
respectively, for better readability.

\begin{proposition}\label{tilde small}
There is an $\mathfrak S_n$-equivariant small projective toric resolution
\[
\tilde \mu ^{(n)\prime }: \widetilde Z^{(n)\prime}\to \widetilde X^{(n)\prime}.
\]
\end{proposition}

\begin{proof}
$\widetilde X^{(n)\prime}$ is the affine closed subvariety in $\mathbb C^{2n}$ defined by
\begin{equation}\label{d2eq}
z_{11}z_{12}=z_{21}z_{22}=\cdots =z_{n1}z_{n2}.
\end{equation}
The $n$-plane $\Sigma ^{(n)}$ defined by
$z_{11}=z_{21}=\cdots =z_{n1}=0$
is an $\mathfrak S_n$-invariant non-Cartier Weil divisor on $X^{(n)\prime}$.
Let
\[
\tilde f ^{(n)}:\widetilde W^{(n)\prime} \to \widetilde X^{(n)\prime}
\]
be the blowing-up of $\widetilde X^{(n)\prime}$ along $\Sigma ^{(n)}$.
$\widetilde W^{(n)\prime}$ is the closed subvariety of $\mathbb C^{2n}\times \mathbb P^{n-1}$
defined by
\[
z_{i1}y_j-z_{j1}y_i=0\quad (i\neq j)
\]
along with \eqref{d2eq}, where $[y_1:\cdots :y_n]$ is the homogeneous coordinate of $\mathbb P^{n-1}$.
Let $P_i=[0:\cdots :0: \overset{\substack{i\\ \vee}}{1} : 0 : \cdots : 0]\in \mathbb P^{n-1}$
and $U_i=(y_i\neq 0)\subset \mathbb P^{n-1}$.  Then, it is easily checked using coordinate that
there is a natural isomorphism of affine varieties
\[
\widetilde W^{(n)\prime}\cap (\mathbb C^{2n}\times U_i)\cong
\widetilde X^{(n-1)\prime}\times \mathbb C.
\]
Moreover,
\[
D_i = \widetilde W^{(n)\prime}\cap (\mathbb C^{2n}\times \{P_i\}) \quad (i=1,2,\cdots, n)
\]
is a non-$\mathbb Q$-Cartier Weil divisor of $\widetilde W^{(n)\prime}$ that is
the strict transform of the divisor on $\widetilde X^{(n)\prime}$ defined by
\[
z_{11}=\cdots =z_{i-1,1}=z_{i2}=z_{i+1,1}=\cdots =z_{n1}=0.
\]
$D_i$ is identified with
$\Sigma ^{(n-1)}\times \mathbb C\subset \widetilde X^{(n-1)\prime}\times \mathbb C$ under
the isomorphism above. $D_i$'s are disjoint to each other and the union
$D=\amalg \, D_i$ is $\mathfrak S_n$-invariant. Therefore, the blowing-up
of $\widetilde W^{(n)\prime}$ along $D$ is $\mathfrak S_n$-equivariant.
As it is locally isomorphic to $\tilde f ^{(n-1)}$, we get an $\mathfrak S_n$-equivariant
small resolution of $\widetilde X^{(n)\prime}$ by induction on $n$.
The centers of the blowing-ups are
strict transform of torus invariant (non-$\mathbb Q$-Cartier) divisors on
$\widetilde X^{(n)\prime}$, so the resolution is a toric morphism.
\end{proof}

\begin{remark}
The toric variety $\widetilde Z^{(n)\prime}$ also appeared in \cite{W} as
the local model of `augmented relative Hilbert scheme'.
\end{remark}

\subsection{}
From this proposition, one immediately sees that the relative self-product
$\widetilde X^{(n)}$
admits an $\mathfrak S_n$-equivariant
small projective toric resolution
\[\tilde \mu ^{(n)}=(\tilde \mu ^{(n)\prime}\times \id):
\widetilde Z^{(n)} = \widetilde Z^{(n)\prime}\times \mathbb C^n
\to \widetilde X^{(n)}=\widetilde X^{(n)\prime}\times \mathbb C^n.
\]
Now we let $Z^{(n)}=\widetilde Z^{(n)}/\mathfrak S_n
=(\widetilde Z^{(n)\prime} \times \mathbb C^n)/\mathfrak S_n$.
Then, we get a small projective birational morphism
\[
\mu ^{(n)}: Z^{(n)}\to X^{(n)}
\]
and an induced family
\[
\rho ^{(n)}=\pi ^{(n)} \circ \mu ^{(n)} : Z^{(n)}\to B.
\]
We want to study the singular locus of $Z^{(n)}$. For that purpose,
we need a description of the toric birational morphism
$\widetilde \mu ^{(n)\prime}: \widetilde Z^{(n)\prime}\to \widetilde X^{(n)\prime}$
in terms of fan.

\subsection{}
The blowing-up $\tilde f^{(n)}$ appeared above corresponds to
the star subdivision (see \cite{CLS}, \S11.1 for the definition)
$\Theta^{(n)}$ of the cone $\sigma ^{(n)}$ with respect to the ray
spanned by ${}^t(1,0,\cdots, 0)$. The fact that
the ray is a one dimensional face of $\sigma ^{(n)}$ corresponds to
the smallness of $\tilde f^{(n)}$.
The cone $\sigma ^{(n)}$ is spanned by the column vectors of
the $(n+1)\times 2^n$ matrix $C^{(n)}$. one can check that the resolution
$\tilde \mu ^{(n)\prime}: \widetilde Z^{(n)\prime}\to \widetilde X^{(n)\prime}$
is given by the consecutive star subdivisions
of $\sigma ^{(n)}$ with respect
to first $(2^n-1)$ column vectors (in this order).

Let $\Delta ^{(n)}$ be the resulted fan in $N\otimes \mathbb R$ and
$\widetilde Z^{(n)\prime}$ is the toric variety $X(\Delta ^{(n)})$.
By the proof of Proposition \ref{tilde small}, one sees that
$\mathfrak S_n$ acts transitively on the set of open subsets
$\{\widetilde W^{(n)\prime}\cap (\mathbb C^{2n}\times U_i)\}_{i=1}^n$.
This means that $\mathfrak S_n$ acts transitively on the set of maximal cones
$\{\theta _i\}_{i=1}^n$ of the fan $\Theta ^{(n)}$ corresponding to $\widetilde W^{(n)\prime}$.
Actually, $\mathfrak S_n$ acts on it via the permutation of the index set $\{1,2,\cdots, n\}$.
Each cone $\theta _i$ can be identified with $\sigma ^{(n-1)}$ and
its stabilizer subgroup $\Stab (\theta _i)\subset \mathfrak S_n$ is nothing but
the subgroup of permutations that leave $i$ invariant, which is naturally isomorphic to
$\mathfrak S_{n-1}$. An inductive argument infers that the set of maximal cones of
the fan $\Delta ^{(n)}$ consists of $n!$ cones and $\mathfrak S_{n}$ acts on the set
transitively. By the construction, $\Delta ^{(n)}$ contains a cone $\delta ^{(n)}$ that is
generated by the column vectors of
\begin{equation}\label{delta n}
\begin{pmatrix}
1 & 0 & 0 & \cdots & 0\\
0 & 1 & 1 & \cdots & 1\\
0 & 0 & 1 & \cdots & 1\\
 & & & \ddots & \\
0 & 0 & 0 & \cdots & 1
\end{pmatrix} .
\end{equation}

\subsection{}\label{cremona}
Let $\bar \Delta ^{(n)}$ be the Coxeter fan of $A_{n-1}$-root system,
namely the fan whose maximal cones are Weyl chambers of $A_{n-1}$-root system
on the weight lattice $\overline N=\mathbb Z^{n-1}$.
Here, we adopt somewhat non-standard realization of $A_{n-1}$-root system.
Regardless of inner product, we set the non-zero vectors in $\overline N=\mathbb Z^{n-1}$ whose
entries are $0$ or $1$ the positive primitive weight vectors, and
the negative primitive weight vectors are the negation of the positive primitive weight vectors.
Note that this determines an $\mathfrak S_n$-action on the lattice $\overline N=\mathbb Z^{n-1}$,
namely for $k<n-1$
\[
(k\ \ k+1)=
\left(\begin{array}{c|cc|c}
I_{k-1} & & & \\ \hline
& 0 & 1 & \\
& 1 & 0 & \\ \hline
& & & I_{n-k-2}
\end{array}\right)
\mbox{\; and \;}
(n-1\ \ n)=
\begin{pmatrix}
1 & 0 & \cdots & 0 & -1 \\
0 & 1 & \cdots & 0 & -1 \\
\vdots & \vdots & \ddots & \vdots & -1\\
0 & 0 & \cdots & 1 & -1 \\
0 & 0 & \cdots & 0 & -1
\end{pmatrix}.
\]
Now we consider the projective toric variety $X(\bar{\Delta}^{(n)})$.
It has been studied by several authors
\cite{P, Stem, DL} in connection with combinatorics theory. Speaking in a geometric language,
$X(\bar{\Delta}^{(n)})$ is the canonical elimination of indeterminacy
of the standard Cremona transformation of degree $n-1$ in
$\mathbb P^{n-1}$ (\cite{D}, Example 7.2.5).  More precisely, we have a sequence of projective
birational morphisms
\[
g: X(\bar{\Delta}^{(n)})=X_1\overset{g_1}{\lto} X\lto \cdots
\lto X_{n-2}\overset{g_{n-2}}{\lto} X_{n-1}=\mathbb P^{n-1},
\]
where $g_i$ is the blowing-up of the strict transform of the union of the linear subspaces defined
by vanishing of $i+1$ projective coordinates (\cite{DL}, Lemma 5.1).
We take a fan $\Phi$ in $\overline N=\mathbb Z^{n-1}$ spanned by
\[
v_1=\begin{pmatrix} 1 \\ 0 \\  \vdots \\ 0 \end{pmatrix},\;
v_2=\begin{pmatrix} 0 \\ 1 \\  \vdots  \\ 0 \end{pmatrix},\;
\dots ,\;
v_{n-1}=\begin{pmatrix} 0 \\ 0 \\  \vdots \\  1 \end{pmatrix}, \;
v_n=\begin{pmatrix} -1 \\ -1 \\ \vdots \\ -1 \end{pmatrix}.
\]
The fan $\bar \Delta ^{(n)}$ is a subdivision of $\Phi$ and the associated toric morphism
$X(\bar \Delta ^{(n)})\to X(\Phi)=\mathbb P^{n-1}$ is nothing but the above-mentioned birational
morphism $g$.

Let $D_{pos}$ (resp. $D_{neg}$) be the
torus invariant divisor on a toric variety $X(\bar \Delta^{(n)})$ corresponding
to the sum of positive (resp. negative) primitive weight vectors.
We take a homogeneous coordinate $[x_1:\dots : x_n]$ on $\mathbb P^{n-1}$ such that
the prime divisor corresponding to $v_j$ is the hyperplane $(x_j=0)$ for $1\leq j\leq n$.
Then, the $\mathfrak S_n$-action on the toric variety $X(\Phi)$ coincides with the
natural permutation of coordinates on $\mathbb P^{n-1}$:
\[
s\cdot [x_1:\dots:x_n] = [x_{s(1)}:\dots :x_{s(n)}]\quad (s\in \mathfrak S_n).
\]
Under this choice of coordinate, we have $D_{neg}=g^* \div (x_n)$.
Let $\Phi ' $ be another $\mathbb P^{n-1}$-fan on $\overline N$ spanned by
$-v_1,\; -v_2,\; \dots ,\; -v_n$. $\bar \Delta^{(n)}$ is again a subdivision of $\Phi '$ and
let $h:X(\bar \Delta ^{(n)})\to X(\Phi ')=\mathbb P^{n-1}$ be the associated birational morphism.
Then, the composition $\xymatrix{h\circ g^{-1}:\mathbb P^{n-1}\ar@{..>}[r] & \mathbb P^{n-1}}$ is nothing but the
standard Cremona transformation
\[
[x_1:\dots :x_n]\mapsto \left[\frac 1{x_1} : \dots : \frac 1{x_n}\right],
\]
and we have $D_{pos}=h^*\div(x_n)$.
In the below, we sometimes denote the toric variety $X(\bar \Delta^{(n)})$ by $X(A_{n-1})$ for easy recognition.

\begin{proposition}\label{description coxeter}
The toric variety $\widetilde Z^{(n) \prime}=X(\Delta^{(n)})$ is
isomorphic to the total space of a rank 2 vector bundle
\[
\mathcal O(-D_{pos})\oplus \mathcal O(-D_{neg})
\]
over $X(\bar \Delta ^{(n)})=X(A_{n-1})$. Moreover, the projection
\[
\eta ^{(n)}: X(\Delta ^{(n)}) \to X(\bar{\Delta}^{(n)})
\]
is $\mathfrak S_n$-equivariant.
\end{proposition}

\begin{proof}
Let $Q: N\to N$ be an automorphism defined by left multiplication of the matrix
\[
Q=
\begin{pmatrix}
1 & 1 & 0 & \cdots & 0 & -1 \\
0 & 1 & 0 & \cdots & 0 & -1 \\
0 & 0 & 1 & \cdots & 0 & -1 \\
\vdots  & \vdots & \vdots & \ddots & \vdots & \vdots \\
0 & 0 & 0 & \cdots & 1 & -1 \\
0 & 0 & 0 & \cdots & 0 & 1
\end{pmatrix}.
\]
Then, we see that
\begin{equation}\label{sum of line bundles}
Q C^{(n)}_2 =
\left(
\begin{array}{c|cccc|cccc|c}
1 & 1 & \cdots & \cdots & 1 & 0 & \cdots &\cdots & 0 &  0 \\ \hline
0 & &&& & &&& & 0 \\
\vdots & \multicolumn{4}{|c|}{\shortstack{positive primitive\\ weight vectors}} & \multicolumn{4}{|c|}{\shortstack{negative primitive\\ weight vectors}} & \vdots \\
0 & &&& & &&& & 0 \\ \hline
0 & 0 & \cdots & \cdots & 0 & 1 & \cdots & \cdots & 1 & 1
\end{array}
\right).
\end{equation}
In particular, $Q\delta ^{(n)}$ is generated by column vectors of
\[
\left(
\begin{array}{c|ccc|c}
1 & 1 & \cdots & 1 & 0 \\ \hline
0 & 1 & \cdots & 1 & 0 \\
\vdots  &  & \ddots & & \vdots\\
0 & 0 & \cdots & 1 & 0 \\ \hline
0 & 0 & \cdots & 0 & 1
\end{array}
\right).
\]
Let $p:N=\mathbb Z^{n+1} \to \overline N=\mathbb Z^{n-1}$ be the projection
to middle $(n-1)$-factors. From the matrix representation \eqref{Sn action on N},
it is easy to see that the composition
\[
\Pi =p\circ Q:N\to \overline N
\]
is $\mathfrak S_n$-equivariant.
As the maximal cones in $\Delta ^{(n)}$ are $\mathfrak S_n$-translates of $\delta ^{(n)}$ and $\Pi$ is $\mathfrak S_n$-equivariant,
we know that $\Pi$ is compatible with fans $\Delta ^{(n)}$ and $\bar \Delta ^{(n)}$, and induces a toric morphism
$\eta ^{(n)}: X(\Delta ^{(n)})\to X(\bar \Delta ^{(n)})$. Moreover, one sees from \eqref{sum of line bundles} and \cite{CLS} Proposition 7.3.1
that $X(\Delta ^{(n)})$ is the total space of the direct sum of line bundles
$\mathcal O(-D_{pos})\oplus \mathcal O(-D_{neg})$.
\end{proof}

\subsection{}\label{chart coordinate}
The coordinate $[x_1:\dots :x_n]$ on $\mathbb P^{n-1}$ gives a convenient system of local coordinate
on $X(A_{n-1})$ as in the following way. Let $\bar \delta ^{(n)}$ be the positive Weyl chamber generated by
the column vectors of
\[
\begin{pmatrix}
1 & 1 & \cdots & 1 \\
0 & 1 & \cdots & 1 \\
\vdots & \vdots & \ddots & \vdots \\
0 & 0 & \cdots & 1
\end{pmatrix}.
\]
This is a smooth cone and the corresponding affine open subset $U$, which is isomorphic to $\mathbb C^{n-1}$,
has a toric coordinate
\[
\left(\frac{x_1}{x_2},\frac{x_2}{x_3},\dots ,\frac{x_{n-1}}{x_n}\right).
\]
A maximal cone in the $A_{n-1}$-Coxeter fan $\bar \Delta ^{(n)}$ is written as $s\cdot \bar \delta ^{(n)}$
for some $s\in \mathfrak S_n$. It is immediate to see that the corresponding affine open subset $U_s$
has a toric coordinate
\[
\left(\frac{x_{s(1)}}{x_{s(2)}},\frac{x_{s(2)}}{x_{s(3)}},\dots ,\frac{x_{s(n-1)}}{x_{s(n)}}\right).
\]

\subsection{}\label{s fixed open}
Let $\lambda=(\lambda_1,\dots, \lambda_r)$ be a partition of $n$, \emph{i.e.},
let $\lambda$ satisfy
$\lambda _1\geqslant \dots \geqslant \lambda _r>0$ and
$\lambda_1+\dots +\lambda _r=n$. We set
$l_j=\sum _{i=1}^j \lambda _i$
for $1\leqslant j\leqslant r$ (and $l_0=0$ for convenience).
Let $c_j\; (1\leqslant j\leqslant r)$ be a cyclic permutation
\[
c_j=(l_{j-1}+1\; \dots \; l_j)
\]
of length $\lambda _j$ and we take a standard element
\begin{equation}\label{standard lambda element}
\begin{aligned}
s=s_{\lambda} &=(1 \; \dots \; l_1)(l_1+1\; \dots \; l_2)\dots (l_{r-1}+1\; \dots \; l_r)\\
&=c_1c_2 \dots c_r
\end{aligned}
\end{equation}
of $\mathfrak S_n$ in the conjugacy class determined by $\lambda$.

Let $\overline N^{\langle s\rangle}$ be the
sublattice of $s$-fixed vectors.
If we denote the standard basis of $\overline N=\mathbb Z^{n-1}$ by $e_1,\dots, e_{n-1}$,
$\overline N^{\langle s\rangle}$ is generated by
\[
e_1+\dots +e_{l_1},\; e_{l_1+1}+\dots +e_{l_2},\; \dots , \; e_{l_{r-2}+1}+\dots +e_{l_{r-1}},
\]
so that $\overline N^{\langle s \rangle}\cong \mathbb Z^{r-1}$. Every $s$-fixed point on $X(A_{n-1})$ is
contained in an open subset
\[
X_{\overline N}(\bar \Delta ^{(n)}\cap (N^{\langle s\rangle}\otimes \mathbb R))\cong (\mathbb C^*)^{n-r} \times X(A_{r-1}),
\]
which $\mathfrak S_n$-equivalently birationally dominates $(\mathbb C^*)^{n-r}\times \mathbb P^{r-1}$.
Let
\[
t_{ji}=\frac{x_{l_{j-1}+i+1}}{x_{l_{j-1}+i}},\; \mathbf t_j=(t_{j1},\dots , t_{j,\lambda _j-1}), \; \mbox{and} \;
y_k = x_{l_{k-1}+1}.
\]
Then,
\[
\left((\mathbf t_1,\dots, \mathbf t_r),[y_1:\dots :y_r]\right) =
\left( (t_{11},\dots, t_{1,\lambda _1-1}\, ;\, \dots \, ;\, t_{r1},\dots, t_{r,\lambda _r-1}),[y_1:\dots :y_r]\right)
\]
is a coordinate on
$(\mathbb C^*)^{n-r}\times \mathbb P^{r-1}$ and the action of $c_j$ is given by
\begin{multline*}
\left( (\mathbf t_1 ; \, \dots \, ; \;\; t_{j1},\dots, t_{j,\lambda_j-1} \;\; ;\, \dots \,; \;
\mathbf t_r ),[y_1:\dots :y_j:\dots :y_r]\right)\\
\mapsto
\left( (\mathbf t_1 ; \dots ; \; t_{j2},\dots, t_{j,\lambda_j-1}, \frac 1{t_{j1}\dots t_{j,\lambda _j-1}} \; ; \dots \,;\; \mathbf t_r),
[y_1:\dots :t_{j1}y_j:\dots :y_r]\right).
\end{multline*}
In particular , if a point on $(\mathbb C^*)^{n-r} \times X(A_{r-1})$ is fixed by $s$,
we necessarily have
\[
t_{j1}=\dots =t_{j,\lambda _j-1}=\alpha _j
\]
for all $1\leqslant j\leqslant r$,
where $\alpha _j$ is a $\lambda _j$-th root of unity. Let us fix such a point
\[
\boldsymbol{\alpha}=(\alpha_1,\dots, \alpha_1;\dots ;\alpha_r,\dots, \alpha_r)
\in (\mathbb C^*)^{n-r}
\]
and $p\in \mathfrak S_r$. Let $U_p\cong \mathbb C^{r-1}$ be the affine open subset
of $X(A_{r-1})$ corresponding to $p$.
At a point $(\boldsymbol\alpha,\xi)\in (\mathbb C^*)^{n-r}\times U_p$, the acton of $c_j$ is given by
\begin{multline}\label{alpha action parts}
(\boldsymbol\alpha, \left(\frac{y_{p(1)}}{y_{p(2)}},\dots ,\frac{y_{p(r-1)}}{y_{p(r)}}\right) )\\
\mapsto
(\boldsymbol\alpha, \left(\frac{y_{p(1)}}{y_{p(2)}},\dots ,
\frac {1}{\alpha _j}\cdot \frac{y_{p(p^{-1}(j)-1)}}{y_j},
\alpha _j \cdot \frac{y_j}{y_{p(p^{-1}(j)+1)}}, \dots,
\frac{y_{p(r-1)}}{y_{p(r)}}\right) ),
\end{multline}
thus the action of $s=c_1\dots c_r$ is of the form
\begin{equation}\label{alpha action}
(\boldsymbol\alpha, \left(\frac{y_{p(1)}}{y_{p(2)}},\dots ,\frac{y_{p(r-1)}}{y_{p(r)}}\right) )\mapsto
(\boldsymbol\alpha, \left(\frac{\alpha_{p(1)}}{\alpha_{p(2)}}\cdot \frac{y_{p(1)}}{y_{p(2)}},\dots ,
\frac{\alpha_{p(r-1)}}{\alpha_{p(r)}}\cdot \frac{y_{p(r-1)}}{y_{p(r)}}\right) ).
\end{equation}
In particular, for the torus invariant point $\xi _p$, the origin of $U_p\cong \mathbb C^{r-1}$,
$(\boldsymbol\alpha, \xi _p)$ is always an $s$-fixed point.

\subsection{}\label{action on the fiber}
Now we consider the action of $s$ on the fiber $\mathbb C^2$ of  $\mathcal O(-D_{pos})\oplus \mathcal O(-D_{neg})$
over a fixed point. In \S\ref{cremona}, we saw that $D_{neg}=g^*\div (x_n)$.
Noting that $n=l_r$ and $s(n)=l_{r-1}+1$, we have
\[
(s^{-1})^* D_{neg}=g^*\div (x_{l_{r-1}+1}) = D_{neg}+\div\left(\frac{x_{l_{r-1}+1}}{x_n}\right),
\]
hence an isomorphism of invertible sheaves
\[
\xymatrix @=50pt {
(s^{-1})^* \mathcal O (-D_{neg})  \ar[r]^{\ \ \ \ \frac{x_{l_{r-1}+1}}{x_n}} & \mathcal O (-D_{neg}).
}
\]
Let us take an $s$-fixed point $(\boldsymbol \alpha,\xi)\in (\mathbb C^*)^{n-r}\times U_p\; (p\in \mathfrak S_r)$ and
set $\nu =p^{-1}(r)$. Then, it is easy to see that
\[
\mathcal O_{(\mathbb C^*)^{n-r}\times U_p}(-D_{neg})
=\frac{y_{p(\nu+1)}}{y_{p(\nu)}} \dots \frac{y_{p(r)}}{y_{p(r-1)}} \mathcal O_{(\mathbb C^*)^{n-r}\times U_p}.
\]
Therefore, for any $\xi \in U_p$,
the action of $s$ on the fiber $\mathcal O(-D_{neg})\otimes \kappa (\boldsymbol \alpha , \xi)$ is given by the composition
\begin{multline*}
\xymatrix {
\mathcal O(-D_{neg})\otimes \kappa (\boldsymbol \alpha , \xi) \ar[r]^{\frac{\alpha_{p(r)}}{\alpha_r}\qquad}
& (s^{-1})^*\mathcal O(-D_{neg})\otimes \kappa (\boldsymbol \alpha , \xi)
} \\
\xymatrix{
\ar[r]^{\alpha _r\qquad\qquad} & \mathcal O(-D_{neg})\otimes \kappa (\boldsymbol \alpha , \xi),
}
\end{multline*}
namely by a multiplication of $\alpha _{p(r)}$. One also sees by the same argument that the action of $s$ on the fiber
$\mathcal O(-D_{pos})\otimes \kappa (\boldsymbol \alpha , \xi)$ is given by the multiplication of $\alpha _{p(1)}^{-1}$.

\begin{lemma}\label{local eigenvalues}
Let $s\in \mathfrak S_n$ as in \eqref{standard lambda element} and take a fixed point
$q=(q_1,q_2)\in \widetilde Z^{(n)}=X(\Delta ^{(n)})\times \mathbb C^n$.
Assume $\eta ^{(n)}(q_1)=(\boldsymbol \alpha, \xi)\in (\mathbb C^*)^{n-r}\times U_p\subset X(\bar \Delta ^{(n)})$.
Then, the eigenvalues of the Jacobian matrix $J_q(s)$ of $s$ at $q$ is
\begin{multline*}
(\underbrace{
\zeta_1,\dots, \zeta_1 ^{\lambda _1-1}; \; \dots\;  ;
\zeta _r, \dots ,\zeta _r ^{\lambda _r-1}
}_{\mbox{\emph{(i)}}}; \ \ \
\underbrace{
\frac{\alpha _{p(1)}}{\alpha _{p(2)}}, \dots , \frac{\alpha _{p(r-1)}}{\alpha _{p(r)}}
}_{\mbox{\emph{(ii)}}}; \\
\underbrace{
\alpha _{p(1)}^{-1}, \alpha _{p(r)}
}_{\mbox{\emph{(iii)}}}; \ \ \
\underbrace{
1, \zeta_1,\dots, \zeta_1 ^{\lambda _1-1}; \; \dots\;  ;
1, \zeta _r, \dots ,\zeta _r ^{\lambda _r-1}
}_{\mbox{\emph{(iv)}}}
).
\end{multline*}
In particular, we always have $\det \left(J_q(s)\right)=1$.
\end{lemma}

\begin{proof}
The part (i) comes from the action of $s$ on $(\mathbb C^*)^{n-r}$. More precisely,
as we saw in \S\ref{s fixed open}, $s$ acts on $(\mathbb C^*)^{n-r}$ by
\[
(\dots ; \; t_{j1},\dots, t_{j,\lambda_j-1} \; ; \dots )
\mapsto
(\dots ; \; t_{j2},\dots, t_{j,\lambda_j-1}, \frac 1{t_{j1}\dots t_{j.\lambda _j-1}} \; ; \dots ),
\]
therefore the corresponding Jacobian matrix at $t_{ji}=\alpha _j\; (1\leqslant j\leqslant r,\, 1\leqslant i\leqslant \lambda _j-1)$
is a block matrix with components of the form
\[
\begin{pmatrix}
0 &1 & 0 & \cdots & 0\\
0 & 0 & 1 & \cdots & 0 \\
\vdots & \vdots & \vdots & \ddots &\vdots \\
0 & 0 & 0 & \cdots & 1 \\
-1 & -1 & -1 &\cdots & -1
\end{pmatrix},
\]
whose characteristic polynomial is $1+t+\dots +t^{\lambda _j-1}$. This gives the part (i).
The part (ii) comes from the action \eqref{alpha action} of $s$ on $U_p$, and the part (iii) is the action on
the fiber of $\mathcal O(-D_{pos})\oplus \mathcal O(-D_{neg})$ (\S\ref{action on the fiber}).
The part (iv) is the contribution from the permutation representation of $\mathfrak S_n$ on the second factor
$\mathbb C^n$ in $\widetilde Z^{(n)}$.
\end{proof}

\begin{theorem}\label{gorenstein canonical model}
The quotient $Z^{(n)}=\widetilde Z^{(n)}/\mathfrak S_n$ has only Gorenstein canonical quotient singularities
and the singular fiber $(\rho ^{(n)})^{-1}(0)\subset Z^{(n)}$ of the induced family
$\rho ^{(n)}: Z^{(n)}\to B$ is a divisor with $V$-normal crossings
(\cite{S}, Definition (1.16)).
\end{theorem}

\begin{proof}
$V$-normal crossingness automatically follows from the fact that
the morphism $\widetilde Z^{(n)} = X(\Delta ^{(n)})\times \mathbb C^n\to B$ is toric and
each $s\in \mathfrak S_n$ acts on $\widetilde Z^{(n)}$ by a toric morphism.
Therefore, it is enough to show that for every point $q\in \widetilde Z^{(n)}$,
the stabilizer subgroup $\Stab _{\mathfrak S_n}(q)$ is contained in $SL(T_q\widetilde Z^{(n)})$.
This is equivalent to say that for any $s\in \mathfrak S_n$ and $s$-fixed point $q\in \widetilde Z^{(n)}$,
the determinant of the Jacobian matrix $J_q(s)$ for the action of $s$ at $q$ is 1, which is nothing but
the last assertion of the lemma above.
\end{proof}

\begin{remark}
It follows from the theorem that $X^{(n)}$ also has only Gorenstein canonical singularities
(but not $\mathbb Q$-factorial, unlike $Z^{(n)}$).
Since $\mu ^{(n)}:Z^{(n)}\to X^{(n)}$ is small, and therefore $K_{Z^{(n)}}=\mu ^{(n)\, *}K_{X^{(n)}}$,
it is clear that $X^{(n)}$ has only canonical singularities.
The following argument to show that $X^{(n)}$ is Gorenstein
is suggested by the referee;
Let $D$ be an effective divisor such that $-D$ is $\mu ^{(n)}$-ample (\cite{KM}, Lemma 6.28).
Then, for a sufficiently small positive rational numver $\varepsilon$, the pair $(Z^{(n)}, \varepsilon D)$ is klt
and $\mu ^{(n)}$ is a contraction of $(K_{Z^{(n)}}+\varepsilon D)$-negative extremal curves.
Cone theorem (\cite{KM}, Theorem 3.25) implies that there exists a Cartier divisor $B$ on $X^{(n)}$ such that
$K_{Z^{(n)}}=\mu ^{(n)\, *}B$ since $(K_{Z^{(n)}}\cdot C)=0$ for every curve $C$ that is contracted by $\mu ^{(n)}$.
This shows that $K_{X^{(n)}}\sim _{\mathbb Q} B$
and therefore $K_{X^{(n)}}$ is Cartier.
\end{remark}

\section{stringy $E$-polynomial}

\subsection{}\label{stringy E}
The \emph{stringy $E$-function} is a cohomological invariant defined for varieties with only log terminal singularities.
We review the formula, which can be seen as a definition in our purpose, of the stringy $E$-function for
(global) quotient variety. For the foundation of the theory of stringy $E$-functions, we refer \cite{Ba,Y}.

Let $X$ be a variety. By general theory of mixed Hodge structures, the compact support cohomology
$H_c^k(X,\mathbb Q)$ carries a canonical mixed Hodge structure, and hence the \emph{Hodge number}
$h^{p,q}(H^k_c(X))$ is defined. We define the \emph{$E$-polynomial} $E(X)\in \mathbb Z[u,v]$ of $X$ by
\[
E(X)=\sum _{p,q,k} (-1)^k\ h^{p,q}(H^k_c(X))\ u^pv^q.
\]
As customary, we denote the $E$-polynomial of the affine line $\mathbb A^1$ by $\mathbb L$:
\[
\mathbb L= E(\mathbb A^1)=uv.
\]
If $X$ is a toric variety, $X$ is stratified into tori of various dimensions in Zariski topology, therefore,
$E(X)$ can be written as a polynomial in $\mathbb L$.

Let $M$ be a non-singular algebraic variety of dimension $n$ and $G$ a finite group acting on $M$.
We denote a set of complete representatives of the conjugacy classes of $G$ by $\Conj (G)$.
Let $F_g$ be the locus of $g$-fixed points on $M$ for $g\in \Conj (G)$.  For each point $q\in F_g$,
the Jacobian matrix $J_q(g)$ is diagonalizable. We list its eigenvalues as
\[
\left(e^{2\pi i \theta _1}, \dots, e^{2\pi i \theta _n}\right)\quad (0\leqslant \theta _j< 1),
\]
where $\theta _j$ is a rational number whose denominator is a divisor of the order of $g$.
We define the \emph{age} (or \emph{shift number}) of $g$ at $q$ by
\[
\age (g; q) = \sum _{j=1}^n \theta _j.
\]
The age gives a locally constant function
\[
\age : F_g\to \mathbb Q.
\]
For  $\nu \in \mathbb Q$ , we define $F_{g,\nu}=\age ^{-1} (\nu)$.
The age of $g$ is an integer if and only of $J_q(g)\in SL(T_qM)$.

From now on, let us assume $J_q(g)\in SL(T_qM)$ for all $g\in G$ and $q\in M$, namely
we assume $M/G$ is Gorenstein canonical. We also assume that
the quotient map $M\to M/G$ has no ramification divisor, \emph{i.e.},
$\codim F_g>1$ for any $g\in G$.
Then, the \emph{stringy $E$-polynomial
of the quotient variety} $M/G$ is given by
\begin{equation}\label{orb=st}
E_{st}(M/G) = E(M/G) + \!\!\! \sum _{\substack{\id\neq g\in \Conj(G) \\ \nu \in \mathbb Z}}
\!\!\! E\left(F_{g,\nu}/Z(g)\right)\cdot \mathbb L^{\nu},
\end{equation}
where $Z(g)$ is the centralizer of $g\in G$.
The right hand side is called the \emph{orbifold $E$-function} of $M/G$  (see \cite{Ba}, Definition 6.3),
which is known to be the same as the stringy $E$-function that, in turn, is defined in a quite different way
(\cite{Ba} Theorem 7.5). The summands other than $E(M/G)$ are called \emph{twisted sectors},
while we call $E(M/G)$ the \emph{untwisted sector}.

\subsection{}\label{untwisted sector}
Now,  let us move on to the case $M=\widetilde Z^{(n)}$ and $G=\mathfrak S_n$.
The untwisted sector $E(\widetilde Z^{(n)}/\mathfrak S_n)=E(Z^{(n)})$ can be calculated by
the character formula for the cohomology group $H^*(X(A_{n-1}),\mathbb Q)$ due to Procesi, Dolgachev-Lunts, and Stembridge.

\begin{lemma}
$E(X(A_{n-1})/\mathfrak S_n)=(1+\mathbb L)^{n-1}$.
\end{lemma}

\begin{proof}
This is essentially Theorem 3.1 of \cite{Stem}, which states that the $\mathfrak S_n$-invariant part of the
cohomology ring $H^*(X(A_{n-1}))^{\mathfrak S_n}$ has a basis
$\left\{y_J \; |\; J\subset \{\mbox{ simple roots } \} \right\}$ indexed by all the subset of the set of simple roots.
Moreover the degree of $y_J$ is $2\cdot |J|$. Therefore, $H^{2k}(X(A_{n-1})) ^{\mathfrak S_n}$ is of dimension
$\binom{n-1}{k}$. Since $X(A_{n-1})$ is a smooth toric variety, the whole $H^{2k}$ has Hodge type $(k,k)$
(\cite{CLS} Theorem 12.5.3), and we obtain
\[
E(X(A_{n-1}))/\mathfrak S_n)=\sum _{k=0} ^{n-1} \binom{n-1}{k} (uv)^k = (1+\mathbb L)^{n-1}.
\]
\end{proof}

\begin{proposition}
$E(Z^{(n)})=\mathbb L^{n+2}(1+\mathbb L)^{n-1}$.
\end{proposition}

\begin{proof}
We recall that $\widetilde Z^{(n)} = X(\Delta ^{(n)})\times \mathbb C^n$ and $X(\Delta ^{(n)})$ is
the total space of a rank 2 vector bundle over $X(A_{n-1})$. Poincar\'e duality implies
\[
H^{2k+4}_c(X(\Delta ^{(n)})) \cong H^{2k} (X(A_{n-1})) \otimes H^4_c(\mathbb C^2)
\]
and therefore
\[
H^{2k+2n+4}_c(\widetilde Z^{(n)}) \cong H^{2k} (X(A_{n-1})) \otimes H^4_c(\mathbb C^2)\otimes H^{2n}_c(\mathbb C^n).
\]
The 1-dimensional space $H^{2m}_c(\mathbb C^m)$ is spanned by the fundamental class so that a finite group action
always leaves it invariant. Thus we get
\[
E(\widetilde Z^{(n)}/\mathfrak S_n) = E(X(A_{n-1})) \cdot \mathbb L^{n+2} = \mathbb L^{n+2}(1+\mathbb L)^{n-1}
\]
by the previous lemma.
\end{proof}

\subsection{}
For a subset $M\subset \{1,\dots, r\}$, we define $\mathfrak S M$ to be the subgroup consisting of
elements in $\mathfrak S_r$ that leave each element of $\{1,\dots, r\}\backslash M$ invariant.
Let
\[
\varphi : \{1,\dots, r\}\to T
\]
be a map to a set $T$ and $\{\beta _1,\dots , \beta _k\}$
the set of its values.
Taking the level sets $M(\varphi) _j = \varphi ^{-1}(\beta _j)$\; $(1\leqslant j\leqslant k)$,
we get a partition $M(\varphi)=\{M(\varphi)_i\}$,
\[
\{1,\dots, r\}=\coprod _{j=1}^k  M(\varphi )_j\, ,
\] and
\[
m(\varphi) = (|M(\varphi)_1|,\dots, |M(\varphi )_k|)
\]
is a (not necessarily non-increasing) partition
of $r$. We will call $M(\varphi)$ (or $m(\varphi)$) the \emph{multiplicity partition} of $\varphi$.
We define
\[
\mathfrak S_{M(\varphi)}=\mathfrak SM(\varphi)_1 \times
\dots \times \mathfrak SM(\varphi)_k \subset \mathfrak S_r.
\]
This is a Young subgroup of $\mathfrak S_r$.

\subsection{}
Let $\lambda=(\lambda _1,\dots, \lambda _r)$ be a length $r$ partition of $n$.
It defines a non-increasing map $\lambda : \{1,\dots ,r\}\to \mathbb Z$
by $j\mapsto \lambda _j$, and therefore we have the associated multiplicity partition $M(\lambda)$ and
the Young subgroup $\mathfrak S_{M(\lambda)}$. If $m(\lambda) = (m_1,\dots, m_k)$,
we have
\[
M(\lambda )_j = \left\{\sum _{i=1}^j m_{i-1}+1,\dots, \sum _{i=1}^j m_i\right\},
\]
where $m_0=0$ by convention.
For each partition $\lambda=(\lambda _1,\dots, \lambda _r)$, we define
\[
\Theta _{\lambda}
=\{ \theta=(\theta_1,\dots, \theta _r) \; | \;
\theta _j = \frac{a_j}{\lambda _j},\; a_j\in \mathbb Z,\; 0\leqslant a_j < \lambda _j\}
\]
and call an element $\theta \in \Theta _{\lambda}$ an \emph{angle type} associated with $\lambda$.
We say that an angle type $\theta=(\theta_1, \dots ,\theta _r) \in \Theta _{\lambda}$
is \emph{standard} if
\[
\theta _{\sum _{i=1}^j m_{i-1}+1} \leqslant \dots \leqslant \theta _{\sum _{i=1}^j m_i}
\]
for all $1\leqslant j\leqslant k$.
An angle type $\theta$ also determines a map
\[
\theta : \{1,\dots ,r\}\to \mathbb Q\cap [0,1)
\]
and we have the associated multiplicity partition $M(\theta)$ and the Young subgroup
$\mathfrak S_{M(\theta)}$.

\subsection{}
Each angle type $\theta \in \Theta _{\lambda}$ determines a point
$\boldsymbol\alpha(\theta)\in (\mathbb C^*)^{n-r}$ by
\[
\boldsymbol \alpha (\theta) =
(\underbrace{e^{2\pi i \theta _1}, \dots ,e^{2\pi i \theta _1}}_{\mbox{\scriptsize $(\lambda _1-1)$-times}}\ ;\
\dots \ ;\
\underbrace{e^{2\pi i \theta _r}, \dots ,e^{2\pi i \theta _r}}_{\mbox{\scriptsize $(\lambda _r-1)$-times}}
).
\]
Now let $s=s_{\lambda}\in \mathfrak S_n$ be
the standard element in the conjugacy class determined by $\lambda$, as in \eqref{standard lambda element}.
Let $\overline F_s$ be the set of $s$-fixed points on $X(A_{n-1})$. As we saw in \S\ref{s fixed open},
a point in $\overline F_s$ is of the form $(\boldsymbol \alpha(\theta),\xi)\in (\mathbb C^*)^{n-r}\times X(A_{r-1})
\subset X(A_{n-1})$ for some angle type $\theta \in \Theta _{\lambda}$.
Therefore, if we define
\[
\overline F_s ^{\theta} =\overline  F_s \cap \left(\{\boldsymbol \alpha (\theta)\}\times X(A_{r-1})\right),
\]
then we have $\overline F_s=\coprod _{\theta \in \Theta_{\lambda}} \overline F_s^{\theta}$.
We naturally identify $\overline F_s^{\theta}$ with the corresponding closed subset of $X(A_{r-1})$.

The centralizer $Z(s)$ for $s\in \mathfrak S_n$ is generated by
the cyclic permutations $c_1,\dots, c_r$ in the notation of \eqref{standard lambda element}
and the permutations of the cycles of the same length among $\{c_1,\dots ,c_r\}$
(see \cite{Sa}, Proposition 1.1.1). The subgroup $H(s)$ generated by $c_1,\dots, c_r$ in $\mathfrak S_n$
is a normal subgroup of $Z(s)$ and we have $Z(s)/H(s)\cong \mathfrak S_{M(\lambda)}$.

\begin{lemma}
Let $V\subset \overline F_s$ be a union of connected components that is invariant under the action of
the centralizer subgroup $Z(s)$.
Then, $Z(s)$ acts on the cohomology group $H^*(V)$ via a natural action of $\mathfrak S_{M(\lambda)}$,
namely $H^*(V/Z(s))\cong H^*(V/\mathfrak S_{M(\lambda)})$.
\end{lemma}

\begin{proof}
The permutations of the cycles of the same length in $\{c_1,\dots, c_r\}$ act on
$\mathbb P^{r-1}$ through corresponding permutations of homogeneous coordinates $[y_1:\dots :y_r]$,
and accordingly they naturally act on $X(A_{r-1})$.
The action is identified with the natural action of $\mathfrak S_{M(\lambda)}\subset \mathfrak S_r$.
Therefore, it is sufficient to show that the subgroup $H(s)$ acts on $H^*(V)$ trivially.
$H(s)$ leaves $\alpha (\theta)$ invariant and it acts on the torus invariant closed subset $\overline F_s^{\theta}$
as a finite subgroup of the open dense torus associated to $N^{\langle s\rangle}$ of $X(A_{r-1})$.
$H(s)$ also leaves $H^*(\overline F_{s}^{\theta}\cap V)$ invariant, since the cohomology group
$H^*(\overline F_{s}^{\theta}\cap V)$ is generated by the fundamental cycles of torus invariant
closed subvarieties of a smooth projective toric variety $\overline F_{s}^{\theta}\cap V$ (see
\cite{CLS}, Lemma 12.5.1 and Theorem 12.5.3).
\end{proof}

\subsection{}
Let $\varphi :\{1, \dots , r\}\to T$ be a map. We define the adjacency function
\[
\adj(\varphi):\{1, \dots, r\}\to \mathbb Z
\] of $\varphi$ inductively by
\[
\adj(\varphi)(1)=1,\quad
\adj(\varphi)(j)=
\begin{cases}
\adj(\varphi)(j-1) & \mbox{if } \varphi (j)=\varphi (j-1) \\
\adj(\varphi)(j-1)+1 & \mbox{if } \varphi (j)\neq \varphi (j-1)
\end{cases} .
\]
Let $\lambda$ be a length $r$ partition of $n$, $\theta\in \Theta _{\lambda}$ an angle type, and $p\in \mathfrak S_r$.
We define a function
\[
\theta \star p  : \{1,\dots, r\} \to \mathbb Z
\]
by $\theta\star p=\adj(\theta \circ p)\circ p^{-1}$ and let $M(\theta\star p)$ be the corresponding level set
partition of $\{1,\dots, r\}$. The partition defines a Young subgroup
$\mathfrak  S_{M(\theta \star p)}\subset \mathfrak S_r$.
As $M(\theta \star p)$ is a refinement of the partition $M(\theta)$,
$\mathfrak S_{M(\theta\star p)}$ is a subgroup of $\mathfrak S_{M(\theta)}$.
Let $\bar p =\mathfrak S_{M(\theta \star p)} p$ be the right coset in $\mathfrak S_r$
and $P(\theta)=\{\bar p \; |\;  p\in \mathfrak S_r\}$.
Then, it is easy to see that $P(\theta)$ is a partition of $\mathfrak S_r$.
We define
\[
\tau _{\bar p} = \bigcap _{q\in \bar p} q \bar \delta ^{(r)},
\]
where $\bar \delta ^{(r)}$ is the positive Weyl chamber of the $A_{r-1}$-root system
as in \S\ref{chart coordinate}.

\begin{proposition}\label{connected components of fixed point locus}
The set of connected components of $\overline F_s^{\theta}$ agrees with
the set of orbit closures $\{V(\tau _{\bar p})\; |\; \bar p\in P(\theta)\}$. Moreover, if
$m(\theta \star p) = (r_1,\dots, r_k)$, then
$V(\tau _{\bar p}) \cong X(A_{r_1-1})\times \dots \times X(A_{r_k-1})$.
\end{proposition}

\begin{proof}
A codimension one face of $p\bar \delta ^{(r)}$ is cut out by a hyperplane of invariant vectors under
a transposition $(p(j)\;\; p(j+1))$ for some $1\leqslant j<r$.
The face corresponds to an affine line with coordinate
$y_{p(j)}/y_{p(j+1)}$. Taking the action \eqref{alpha action} into account,
this affine line consists of $s$-fixed point if and only if $\theta _{p(j)}=\theta _{p(j+1)}$.
On the other hand, it is easy to verify that
\[
\mathfrak S_{M(\theta \star p)}=\left\langle
(p(j)\;\; p(j+1))\; | \; \theta _{p(j)}=\theta _{p(j+1)}\right\rangle \subset \mathfrak S_r
\]
and if $q\in \bar p$, then $\mathfrak S_{M(\theta \star q)}=\mathfrak S_{M(\theta \star p)}$ as
subgroups of $\mathfrak S_r$. By construction, the orbit $O(\tau _{\bar p})$ is a torus with the coordinates
\[
\left\{\frac{y_{p(j)}}{y_{p(j+1)}}\; \bigg| \; \theta _{p(j)}=\theta _{p(j+1)}\right\}.
\]
In particular, every point in $O(\tau _{\bar p})$, and therefore of $V(\tau _{\bar p})$, is fixed by $s$,
namely $V(\tau _{\bar p})$ is a connected component of $\overline F_s^{\theta}$.

On the other hand, since $s$ acts on $X(A_{r-1})$ via a cyclic subgroup of the torus, every connected
component $V$ of $\overline F_s^{\theta}$ is a torus invariant closed subset of $X(A_{r-1})$.  In particular,
$V$ contains a torus invariant point $\xi _p$ corresponding to a maximal cone $p\bar \delta ^{(r)}$ for
some $p\in \mathfrak S_r$. The above argument shows that $V(\tau _{\bar p})$ is
the connected component of $\overline F_s^{\theta}$ containing $\xi _p$. Hence we know that $V=V(\tau _{\bar p})$.

The image fan on
$\overline N^{\langle s\rangle} / \left(\langle (\tau _{\bar p})\rangle _{\mathbb R}\cap \overline N^{\langle s \rangle}\right)$
corresponding to $V(\tau _{\bar p})$ is
the Coxeter complex of the root system corresponding to the Young subgroup
$\mathfrak S_{M(\theta \star p)}\cong \mathfrak S_{r_1}\times \dots \times \mathfrak S_{r_k}$,
where $m(\theta \star p)=(r_1,\dots, r_k)$.
It immediately follows that $V(\tau _{\bar p})\cong X(A_{r_1-1})\times \dots \times X(A_{r_k-1})$.
\end{proof}

\subsection{}\label{definition of F}
We denote the connected component $V(\tau _{\bar p})\subset \overline F_s^{\theta}$ by
$\overline F_s^{\theta ,\bar p}$.
Let $F_s$ be the closed subset of $s$-fixed points on $X(\Delta ^{(n)})$ and
$F_s^{\theta,\bar p}\subset X(\Delta ^{(n)})$ the union of connected components
of $F_s$ that map to $\overline F_s^{\theta, \bar p}$.
If we define
$\phi(\theta, \bar p)$ to be the number of $0$ in the set $\{\theta _{p(1)},\, \theta _{p(r)}\}$,
the calculation in \S\ref{action on the fiber} implies that
$F_s^{\theta ,\bar p}\to \overline F_s^{\theta ,\bar p}$ is a vector bundle of rank $\phi (\theta, \bar p)$.
In particular $F_s^{\theta ,\bar p}$ is connected. We know that the $s$-fixed point locus of $\widetilde Z^{(n)}$
is a disjoint union of $F_s^{\theta ,\bar p}\times (\mathbb C^n)^{\langle s\rangle}\subset X(\Delta ^{(n)})\times \mathbb C^n
= \widetilde Z^{(n)}$.

\begin{lemma}\label{value of age}
The value of the age function at a point in $F_s^{\theta ,\bar p} \times (\mathbb C^n)^{\langle s\rangle}$ is
given by
\begin{equation}\label{age function}
a(s;\theta, \bar p) = n -r + \sum _{j=1}^{r-1} \left\{ \theta _{p(j+1)} - \theta _{p(j)} \right\}
+ \{1-\theta _{p(1)}\}+ \theta _{p(r)},
\end{equation}
where $\{t\}=t-\lfloor t \rfloor$ is the fractional part of $t$.
\end{lemma}

\begin{proof}
This is just a consequence of Lemma \ref{local eigenvalues}, noting that the contribution from
the parts (i) and (iv) sum up to
\[
\sum _{j=1}^r\sum _{i=1}^{\lambda _j-1} 2\cdot \frac{i}{\lambda _j}
=\sum _{j=1}^r \lambda _j-1 = n -r .
\]
\end{proof}

\subsection{}
Let $\Theta _{\lambda} ^{st}$ be the set of standard angle types.
$\mathfrak S_{M(\lambda)} \subset \mathfrak S_r$ acts on the set of angle types $\Theta _{\lambda}$
by permutation of the factors and each orbit contains a unique
standard element $\theta$. Therefore,
there is an identification $\Theta _{\lambda} / \mathfrak S_{M(\lambda)}\cong \Theta _{\lambda}^{st}$.
The stabilizer subgroup of $\theta$ is $\mathfrak S_{M(\lambda)}\cap \mathfrak S_{M(\theta)}$.
Therefore, we get
\[
\overline F_s / \mathfrak S_{M(\lambda)}
= \left(\coprod _{\theta \in \Theta _{\lambda}} \overline F_s^{\theta} \right) / \mathfrak S_{M(\lambda)}
\cong
\coprod _{\theta \in \Theta _{\lambda}^{st}} \overline F_s^{\theta} / (\mathfrak S_{M(\lambda)}\cap \mathfrak S _{M(\theta)}).
\]
Let $\theta \in \Theta _{\lambda}^{st}$. Then, Proposition \ref{connected components of fixed point locus} says
that we have a decomposition into connected components
\[
\overline F_{s}^{\theta} = \coprod _{\bar p\in P(\theta)} \overline F_s^{\theta ,\bar p},
\]
where $P(\theta)=\{\bar p = \mathfrak S_{M(\theta \star p)} p\; | \; p\in \mathfrak S_r\}$,
and the isomorphism class of
$\overline F_s^{\theta ,\bar p}=V(\tau _{\bar p})$ is determined only by the multiplicity partition $m(\theta \star p)$.
$\mathfrak S_{M(\lambda)}\cap \mathfrak S_{M(\theta)}$ naturally acts on $P(\theta)$.
Let $\overline P(\theta)$ be a complete system of representative for the quotient set
$P(\theta)/(\mathfrak S_{M(\lambda)}\cap \mathfrak S_{M(\theta)})$. Then, we get
\[
\overline F_s^{\theta} /(\mathfrak S_{M(\lambda)}\cap \mathfrak S_{M(\theta)})
\cong \coprod _{\bar p\in \overline P(\theta)} \overline F_s^{\theta ,\bar p}
/(\mathfrak S_{M(\lambda)}\cap \mathfrak S_{M(\theta \star p)}).
\]

Combining what we have obtained above, finally we get the following

\begin{theorem}\label{formula for stringy E-polynomial}
We keep the notation above.
The stringy $E$-polynomial of $Z^{(n)}$ is given by the following formula:
\begin{multline}\label{the formula}
E_{st}(Z^{(n)}) = \mathbb L^{n+2}(1+\mathbb L)^{n-1} \\
+ \sum _{\substack{\lambda \vdash n \\ \lambda \neq (1^n)}}\sum _{\theta \in \Theta _{\lambda}^{st}}
\sum _{\bar p \in \overline P(\theta)}
E(\overline F _{s_{\lambda}} ^{\theta,\overline p}/(\mathfrak S_{M(\lambda)}\cap \mathfrak S_{M(\theta\star p)}))
\cdot \mathbb L^{\phi (\theta , \bar p)+r(\lambda)+a(s; \theta,\bar p)},
\end{multline}
where $s_\lambda$ is the standard permutation associated with $\lambda$ as in \eqref{standard lambda element}
and $r(\lambda)$ stands for the length of the partition $\lambda$.
\end{theorem}

\subsection{} \label{use of PDLS}
We remark that one can actually calculate the term
$E(\overline F _{s_{\lambda}} ^{\theta,\overline p}/(\mathfrak S_{M(\lambda)}\cap \mathfrak S_{M(\theta\star p)}))$
in the formula \eqref{the formula} as follows.

We saw in Proposition \ref{connected components of fixed point locus} that
\[
\overline F_s^{\theta, \bar p}\cong X(A_{r_1-1})\times \dots \times X(A_{r_k-1})
\]
if $m(\theta \star p)=(r_1,\dots, r_k)$. Since $\mathfrak S_{M(\lambda)} \cap \mathfrak S_{M(\theta \star p)}$
is a Young subgroup of
\[
\mathfrak S_{M(\theta \star p)}\cong \mathfrak S_{r_1}\times \dots \times
\mathfrak S_{r_k},
\]
there is a partition $\mu _j\vdash r_j$ for each $1\leqslant j\leqslant k$ such that
\[
\mathfrak S_{M(\lambda)} \cap \mathfrak S_{M(\theta \star p)}
\cong \mathfrak S_{\mu _1}\times \dots \times \mathfrak S_{\mu _k},
\]
where $\mathfrak S_{\mu _j}\subset\mathfrak S_{r_j}$ is the Young subgroup
associated with the partition $\mu _j\vdash r_j$. Therefore, we have
\[
E(\overline F _{s_{\lambda}} ^{\theta,\overline p}/(\mathfrak S_{M(\lambda)}\cap \mathfrak S_{M(\theta\star p)}))
= E(X(A_{r_1-1})/\mathfrak S_{\mu _1})\times \dots \times E(X(A_{r_k-1})/\mathfrak S_{\mu _k}).
\]

On the other hand, $E(X(A_{r-1})/\mathfrak S_{\mu})$ for $\mu \vdash r$ can be calculated
by the character formula of Procesi, Dolgachev-Lunts, and Stembridge:
if we define $\chi _l(A_{n-1})$ to be the character of the $\mathfrak S_n$-representation
$H^{2l}(X(A_{n-1}),\mathbb Q))$ and $\chi [A_{n-1},q]=\sum _{l=0}^{n-1} \chi _l(A_{n-1})q^l$ the generating function,
then we know that
\begin{align*}
1+\sum _{n\geqslant 1} \chi [A_{n-1},q]t^n
& =
\frac{1+\sum _{m\geqslant 1} h_m t^m}{1-\sum _{m\geqslant 2}(q+\dots +q ^{m-1}) h_mt^m}  \numberthis \label{PDLS} \\
& =1+ h_1t +h_2(1+q)t^2+(h_3+(h_1h_2+h_3)q+h_3q^2)t^3\\
&\qquad \qquad +(h_4+(h_2^2+h_1h_3+h_4)(q+q^2)+h_4q^3)t^4+\cdots ,
\end{align*}
where $h_m$ is the character of the trivial representation of $\mathfrak S_m$, or rather,
the complete symmetric function of degree $m$ (see \cite{Stem}, Theorem 6.2, see also \cite{P,DL}).
Therefore, if $\mu =(m_1,\dots, m_k)$, the character inner product gives the formula
\[
E(X(A_{r-1}/\mathfrak S_{\mu})) = ( h_{\mu}, \chi[A_{r-1},\mathbb L] ),
\]
where $h_{\mu}=h_{m_1}\cdots h_{m_k}$.

\subsection{} \label{cohomological substitute}
Let $X$ be a variety with only log terminal singularities. The change of variable formula for
motivic integration (\cite{Ba}, Theorem 3.5, or \cite{Y}, Theorem 66) implies that
if $f:Y\to X$ is a proper birational morphism that is crepant, \emph{i.e.},
$K_Y=f^*K_X$, the stringy $E$-function is invariant: $E_{st}(Y)=E_{st}(X)$
(\cite{Ba}, Theorem 3.8). In particular, if $f:Y\to X$ is a crepant resolution,
we have $E_{st}(X)=E_{st}(Y)=E(Y)$.

In our situation, as the birational morphism $\mu ^{(n)}:Z^{(n)}\to X^{(n)}$ is small,
it is in particular crepant. Thus we have $E_{st}(X^{(n)})=E_{st}(Z^{(n)})$.
If the symmetric product $X^{(n)}$ admits a crepant \emph{resolution}
$Y\to X^{(n)}$, we have $E(Y)=E_{st}(Z^{(n)})$.
Moreover if the induced family $Y\to B$ is semistable, and if we denote the
singular fiber by $Y_0$, Poincar\'e duality and homotopy invariance implies
\[
H^p_c (Y)^* \cong H^{2n-p+2}(Y)\cong H^{2n-p+2}(Y_0).
\]
Therefore, the polynomial $E(Y)=E_{st}(Z^{(n)})$ encodes
the cohomological information of the singular fiber of a semistable model.
In general, $X^{(n)}$ may not admit a crepant resolution; nevertheless
the Gorenstein canonical orbifold model $Z^{(n)}$ is already a good substitute
for a minimal semistable model of $X^{(n)}$ on the level of cohomology.

\begin{proposition}\label{stringy E-poly for small n}
The stringy $E$-polynomial $E_{st}(Z^{(n)})$ for $n=2,3,4,5$ is as follows:
\begin{align*}
E_{st}(Z^{(2)}) &= \mathbb L^5+2\mathbb L^4+\mathbb L^3, \\
E_{st}(Z^{(3)}) &= \mathbb L^7+3\mathbb L^6+5\mathbb L^5+2\mathbb L^4, \\
E_{st}(Z^{(4)}) &= \mathbb L^9+4\mathbb L^8+11\mathbb L^7+14\mathbb L^6+4\mathbb L^5, \\
E_{st}(Z^{(5)}) &= \mathbb L^{11}+5\mathbb L^{10}+17\mathbb L^9+35\mathbb L^8+30\mathbb L^7+6\mathbb L^6.
\end{align*}
\end{proposition}

\begin{proof}
We demonstrate the case $n=4$. The other cases are similar (but much more complicated in the case $n=5$).
Nontrivial partition $\lambda$ of $4$ is one of $\lambda = (2,1^2),\, (2^2),\, (3,1),\, (4)$.  We calculate
the contributions to twisted sectors case by case. We represent $p\in \mathfrak S_r$ by a sequence
$p=[p(1), p(2), \dots , p(r)]$ to make it short.

\noindent \underline{Case $\lambda = (2,1^2)$}. The length of the partition is $r=3$. The set of standard angle type in this case is
$\Theta _{(2,1^2)}^{st} = \{ (0^3),\, (1/2,0^2)\}$.
We also note that $\mathfrak S_{M(\lambda)}=\mathfrak S\{2,3\}$.
\begin{description}\setlength{\leftskip}{-1.5em}
 \item[Case $\theta = (0^3)$] In this case, we have $\mathfrak S_{M(\theta)}=\mathfrak S_{M(\theta \star p)}=\mathfrak S_3$
 for all $p\in \mathfrak S_3$. It immediately follows that $P(\theta)=\{\bar 1\}$, $\overline F_{(1\; 2)}^{(0^3)} = X(A_2)$, and
 $\mathfrak S_{M(\lambda)}\cap \mathfrak S_{M(\theta \star 1)}=\mathfrak S\{2,3\}$. We need to know
 $E(X(A_2)/\mathfrak S\{2,3\})$, which is calculated by the argument in \S\ref{use of PDLS} as follows:
 \begin{align*}
 E(X(A_2)/\mathfrak S\{2,3\})
 & =(h_1h_2,\chi [A_2,\mathbb L]) \\
 & =(h_1h_2, h_3+(h_1h_2+h_3)\mathbb L+h_3\mathbb L^2)\\
 & =1+3\mathbb L+\mathbb L^2.
 \end{align*}
 Noting $\phi =2,\, a=1$,
 the associated twisted sector is $\mathbb L^8+3\mathbb L^7+\mathbb L^6$.
 \item[Case $\theta = (1/2,0^2)$] $\mathfrak S_{M(\theta)}=\mathfrak S\{2,3\}$.
 $P(\theta)$ consists of
 \[
 \overline{[1,2,3]}=\overline{[1,3,2]},
 \overline{[2,1,3]}, \overline{[3,1,2]}, \mbox{ and }\overline{[2,3,1]}=\overline{[3,2,1]}.
 \]
 However, $\overline{[2,1,3]}$ and $\overline{[3,1,2]}$ is in the same oribt under the action of
 $\mathfrak S_{M(\lambda)}\cap \mathfrak S_{M(\theta)}=\mathfrak S\{2,3\}$, therefore we have
 $\overline P(\theta) = \{ \bar 1,\, \overline{(1\; 2)},\, \overline{(1\; 2\; 3)}\}$.
 It is straightforward to get the following table:
 \begin{center}
 \begin{tabular}{c || c|c|c|c || c}
 $p$ & $\mathfrak S_{M(\theta\star p)}$ & $\theta \circ p$ & $\phi$ & $a$ & twisted sector \\ \hline \hline
 $1$ & $\mathfrak S\{2,3\}$ & $(1/2,0,0)$ & $1$ & $2$ & $(1+\mathbb L)\mathbb L^6$ \\ \hline
 $(1\; 2)$ & $\{1\}$ & $(0,1/2,0)$ & $2$ & $2$ & $\mathbb L^7$ \\ \hline
 $(1\; 2\; 3)$ & $\mathfrak S\{2,3\}$ & $(0,0,1/2)$ & $1$ & $2$ & $(1+\mathbb L)\mathbb L^6$
 \end{tabular}
 \end{center}
 In total, the contribution is $3\mathbb L^7+2\mathbb L^6$.
\end{description}

\noindent \underline{Case $\lambda = (2^2)$}. We have $r=2,\, \mathfrak S_{M(\lambda)}=\mathfrak S_2$,
and
\[
\Theta _{(2^2)}^{st}=\{(0,0),\, (0,1/2),\, (1/2,1/2)\}.
\]
\begin{description}\setlength{\leftskip}{-1.5em}
 \item[Case $\theta = (0^2)$] $\mathfrak S_{M(\theta)}=\mathfrak S_{M(\theta \circ p)}=\mathfrak S_2$ for
 all $p\in \mathfrak S_2$, and $P(\theta )=\{\bar 1\}$ as before. $\overline F_{(1\; 2)(3\; 4)}^{(0^2)} = X(A_1)$
 and $\mathfrak S_{M(\lambda)}\cap \mathfrak S_{M(\theta)}=\mathfrak S_2$,
 $\phi=2$, and $a=2$, so that the associated twisted sector is
 $E(X(A_1)/\mathfrak S_2)\mathbb L^6 = \mathbb L^7+\mathbb L^6$.
 \item[Case $\theta = (0,1/2)$] Since $\mathfrak S_{M(\theta )}=\{1\}$, $P(\theta )=\{\bar 1,\, \overline{(1\; 2)}\}$
 and $\overline F_{(1\; 2)(3\; 4)}^{(0^2)}$ consists of two points. As we have $\phi =1,\, a=3$ in both cases,
 the associated twisted sector is $2\mathbb L^6$.
  \item[Case $\theta = (1/2,1/2)$] $\mathfrak S_{M(\theta)}=\mathfrak S_{M(\theta \star p)}=\mathfrak S_2$,
  and $P(\theta)=\{\bar 1\}$. As before, we know $\overline F_{(1\; 2)(3\; 4)}^{(1/2,1/2)} = X(A_1)$
 and $\mathfrak S_{M(\lambda)}\cap \mathfrak S_{M(\theta)}=\mathfrak S_2$. Since $\phi =0, a=3$
 the associated twisted sector is $ E(X(A_1)/\mathfrak S_2)\mathbb L^5 = \mathbb L^6+\mathbb L^5$.
\end{description}

\noindent \underline{Case $\lambda = (3,1)$}. In this case, $r=2,\, \mathfrak S_{M(\lambda)}=\{1\}$,
$\Theta _{(3,1)}^{st} = \{ (0^2),\, (1/3,0),\, (2/3,0)\}$.
\begin{description}\setlength{\leftskip}{-1.5em}
 \item[Case $\theta = (0^2)$] As before, $\mathfrak S_{M(\theta)}=\mathfrak S_{M(\theta \star p)}=\mathfrak S_2$
 and $P(\theta) =\{1\}$. Therefore, $\overline F_{(1\; 2\; 3)}^{(0^2)}=X(A_1)$. As $\phi=2, a=2$,
 the twisted sector is $\mathbb L^7+\mathbb L^6$.
 \item[Case $\theta = (1/3,0)$] $\mathfrak S_{M(\theta)}=1$, $P(\theta) = \{\bar 1,\, \overline{(1\; 2)}\}$ and
 $\overline F_{(1\; 2\; 3)}^{(1/3,0)}$ consists of two points. As $\phi =1, a=3$ in this case,
 the twisted sector is $2\mathbb L^6$.
 \item[Case $\theta = (2/3,0)$] This case is completely the same as in the case $\theta =(1/3,0)$.
 The twisted sector is $2\mathbb L^6$.
\end{description}

\noindent \underline{Case $\lambda = (4)$}. In this case $r=1$ and we always have $\mathfrak S_{M(\lambda)}
=\mathfrak S_{M(\theta)}=\{1\}$. Therefore $P(\theta)=\{\bar 1\}$ and
$\overline F^{\theta}_{(1\; 2\; 3\; 4)}$ is just one point set.  As we have
\begin{center}
\begin{tabular}{c|| c | c || c}
$\theta$ & $\phi$ & $a$ & twisted sector \\ \hline\hline
$(0)$ & $2$ & $3$ & $\mathbb L^6$ \\ \hline
$(1/4), (2/4), (3/4)$ & $0$ & $4$ & $\mathbb L^5$
\end{tabular}
\end{center}
the twisted sector in total is $\mathbb L^6+3\mathbb L^5$.

Summing up everything, finally we get
\begin{align*}
E_{st}(X^{(4)}) &= \mathbb L^6 (1+\mathbb L)^3 \\
&\qquad +\left( \mathbb L^8+3\mathbb L^7+\mathbb L^6\right)
+\left( 3\mathbb L^7+2 \mathbb L^6\right) +\left( \mathbb L^7+ \mathbb L^6\right)
+2\mathbb L^6 \\
&\qquad + \left( \mathbb L^6+ \mathbb L^5\right)
+\left( \mathbb L^7+\mathbb L^6\right)+2\mathbb L^6+2\mathbb L^6
+ \left( \mathbb L^6+ 3\mathbb L^5\right)\\
&=\mathbb L^9+4\mathbb L^8+11\mathbb L^7+14\mathbb L^6+4\mathbb L^5.
\end{align*}
\end{proof}

\begin{remark}
In comparison with the case of $\Hilb ^n(S)$ for smooth algebraic surface $S$ (see \emph{e.g.} \cite{Nak}),
it is interesting to look for a formula of the generating function
\[
\sum _{n\geqslant 0} E_{st}(Z^{(n)})\cdot t^n .
\]
Unfortunately, due to combinatorial complication in Theorem \ref{formula for stringy E-polynomial},
the author does not yet have a good answer to this question at the time of writing.
\end{remark}

\section{Minimal model}

In this section, we will discuss more birational modifications of $Z^{(n)}$,
in particular minimal models of $Z^{(n)}$.

\begin{theorem}\label{minimal model}
There exists a projective birational morphism $\nu ^{(n)}:Y^{(n)} \to Z^{(n)}$
satisfying the following conditions:
\begin{enumerate}[\rm (i)]
\item $Y^{(n)}$ has only Gorenstein terminal quotient singularities.
\item $\nu ^{(n)}$ is a crepant divisorial contraction, namely $K_{Y^{(n)}}=\nu ^{(n)\, *} K_{Z^{(n)}}$ and
the exceptional set of $\nu ^{(n)}$ is a divisor.
\item Let $\psi ^{(n)}:Y^{(n)}\to B$ be the composition $\rho ^{(n)}\circ \nu ^{(n)}$.
Then, its general fiber is $\Hilb ^n(\mathbb C^*\times \mathbb C)$ and the singular fiber
is a divisor with $V$-normal crossings.
\end{enumerate}
\end{theorem}

Although the existence of a minimal model of $Z^{(n)}$ is a consequence of
the general theory of minimal model program (MMP) of higher dimensional algebraic varieties \cite{BCHM},
here we stick to an explicit construction of a minimal model $Y^{(n)}$ so that we have a good control on the
singularities of the total space and the singular fiber of the resulted minimal model.
The claims (i) and (iii) will not follow from a straightforward application of MMP.

\subsection{}
From the description of \S\ref{symmetric group action}, the natural morphism
$\tilde\rho ^{(n)\prime}:\widetilde Z^{(n)\prime}=X(\Delta ^{(n)})\to B=\mathbb C$ is a toric morphism
associated with the lattice homomorphism
\[
g= (1\; 1\; 0\; \dots \; 0) : N=\mathbb Z^{n+1} \to \mathbb Z.
\]
If we take a basis of $N$ consisting of the column vectors of $Q$ in the proof of
Proposition \ref{description coxeter}, $g$ is represented by a matrix
$(1\; 0\; \dots \; 0\; 1)$. It immediately follows that the primitive generator
for each ray in the fan $\Delta ^{(n)}$, namely the column vectors of
\eqref{sum of line bundles}, has multiplicity one with respect to $g$.
This means that the fiber of $\tilde\rho ^{(n)\prime}$ over the origin
is the union of all the torus invariant divisors of $X(\Delta ^{(n)})$.
It is easy to see from the description \eqref{Sn action on N} the $\mathfrak S_n$-action on $N$ that
its restriction to $\Ker (g)$ is the permutation representation.
Therefore, the restriction of $\tilde \rho ^{(n)\prime} $ to the torus $N\otimes \mathbb C^*\to \mathbb C^*$
is a trivial family of the permutation action on $(\mathbb C^*)^n$, and
the $\mathfrak S_n$-quotient of $N\otimes \mathbb C^*\times \mathbb C^n\to \mathbb C^*$ is a
trivial family of $\Sym ^n(\mathbb C^*\times \mathbb C)$.
It follows that $Z^{(n) \circ}_n=\rho ^{(n)\; -1} (\mathbb C^*)$ has a crepant divisorial resolution
\[
\psi^{(n) \circ}: Y^{(n) \circ}\to Z^{(n) \circ}
\]
that is a family of Hilbert-Chow morphism $\Hilb^n(\mathbb C^*\times \mathbb C)\to \Sym^n(\mathbb C^*\times \mathbb C)$.
We prove that
$\psi ^{(n)\circ}$ extends to a crepant birational morphism $\psi ^{(n)}:Y^{(n)}\to Z^{(n)}$.

\begin{lemma}
Let $s\in \mathfrak S_n$.
The connected component $F^{\theta, \bar p}_s$ (see \S\ref{definition of F}) of the $s$-fixed point locus
in $X(\Delta ^{(n)})$ has intersection with the open dense torus $N\otimes \mathbb C^*$ if and only if
$\theta$ is a zero-sequence $(0^r)$ where $r$ is the length of the partition $\lambda$ associated with $s$.
\end{lemma}

\begin{proof}
Let us assume that $F^{\theta, \bar p}_s$ has a point in common with $N\otimes \mathbb C^*$.
Then, we necessarily have $\phi (\theta,\bar p)=2$, that is, $\theta _p{(1)}=\theta _{p(r)}=0$.
A point $(\boldsymbol\alpha, \left(\frac{y_{p(1)}}{y_{p(2)}},\dots ,\frac{y_{p(r-1)}}{y_{p(r)}}\right) )\in X(A_{n-1})$
is in the open dense torus if and only if $\frac {y_{p(j)}}{y_{p(j+1)}}\neq 0$ for all $j$.
Taking the action map \eqref{alpha action} into account, it follows that $\theta _j- \theta _{j+1}\in \mathbb Z$.
Therefore, we must have $\theta =(0^r)$. The converse also follows from the action map \eqref{alpha action}
and the definition of $F^{\theta, \bar p}_s$.
\end{proof}

\subsection{}
Let $F^0$ be the union of the fixed point loci of trivial angle type $F_s^{(0^{r})}$
for all $s\in \mathfrak S_n$ (here $r$ is the length of
the associated partition $\lambda$ to $s\in \mathfrak S_n$).
Since $F_s^{(0^r)}$ is the total space of a rank $2$ vector bundle
over $\overline F_s^{(0^{r-1})}\cong X(A_{r-1})$,
we know $\dim F_s^{(0^{r})}=r+1$.
In particular $F^0$ is a divisor in $X(\Delta ^{(n)})$.

\begin{lemma}
Let $q=(q_1,q_2)\in \widetilde Z^{(n)}$.
We define
\[
\Stab ^0(q) = \{s\in \Stab (q)\; |\; q_1\in F_s^{(0^r)}\}.
\]
Then,
\begin{enumerate}[\rm (i)]
  \item $\Stab ^0(q)$ is a Young subgroup of $\mathfrak S_n$.
  \item $\Stab ^0(q)$ is a normal subgroup of $\Stab (q)$.
\end{enumerate}
\end{lemma}

\begin{proof}
(i)\; Let $\bar q_1$ be a image of $q_1$ under the composition
\[
 X(\Delta ^{(n)})
 \to X(A_{n-1})\to \mathbb P^{n-1}.
\]
From the description in \S\ref{s fixed open}, one sees that
$q_1$ is an $s$-fixed point of trivial angle type if and only if
its coordinate satisfies the relation
$\displaystyle \frac {x_{s(i)}}{x_i} = 1$
for all $i$ such that $s(i)\neq i$. If $q_1$ is also $t$-fixed point of
trivial angle typle, we have
$\displaystyle \frac{x_{t\circ s(i)}}{x_i} = \frac{x_{t(s(i))}}{x_{s(i)}}\frac {x_{s(i)}}{x_i}=1$,
and hence $\Stab ^0(q)$ is a subgroup of $\Stab (q)$. Also by the characterization above,
one sees that if $s\in \Stab ^0(q)$ has a cycle decomposition
\[
s = (i_1 \,\cdots \, i_{l_1})(i_{l_1+1}\, \cdots \, i_{l_2})\cdots
(i_{l_{r-1}+1}\, \cdots \, i_{l_r}),
\]
all the elements in the subgroup
\[
\mathfrak S\{i_1 \cdots i_{l_1}\} \times
\mathfrak S\{i_{l_1+1} \cdots i_{l_2}\}\times \dots \times
\mathfrak S\{i_{l_{r-1}+1} \cdots i_{l_r}\}
\]
is also contained in $\Stab ^0(q)$. It implies that $\Stab ^0(q)\subset \mathfrak S_n$ is
a Young subgroup.

\noindent (ii)\;
Assume that $s\in \Stab (q)\subset \mathfrak S_n$ has the partition type
$\lambda$ of length $r$. Then, by Lemma \ref{local eigenvalues},
$s\in \Stab ^0(q)$ if and only if
the multiplicity of $1$ in the eigenvalues of the action
of $s$ on $T_q\widetilde Z^{(n)}$ is $2r+1$.
As the eigenvalues are constant in a conjugacy class,
we know that $\Stab ^0(q)$ is a normal subgroup of $\Stab (q)$.
\end{proof}

\subsection{}
Let $q=(q_1,q_1)\in \widetilde Z^{(n)}=X(\Delta ^{(n)})\times \mathbb C^n$
and assume that $q_1\in F^0$.
Then, by Lemma \ref{local eigenvalues}, the action of an element $s\in \Stab ^0(q)$
on the tangent space $T_q\widetilde Z^{(n)}$ has eigenvalues
\begin{multline*}
\qquad \quad(\zeta_1,\dots, \zeta_1 ^{\lambda _1-1}; \; \dots\;  ;
\zeta _r, \dots ,\zeta _r ^{\lambda _r-1};\; \underbrace{1,\dots, 1}_{r+1}\; ; \\
1, \zeta_1,\dots, \zeta_1 ^{\lambda _1-1}; \; \dots\;  ;
1, \zeta _r, \dots ,\zeta _r ^{\lambda _r-1} \; )\qquad \quad
\end{multline*}
if $s$ has partition type $\lambda =(\lambda _1,\dots, \lambda _r)$.
In particular, the character of the representation $\Stab ^0(q)\to GL(T_q \widetilde Z^{(n)})$
agrees with the character of the permutation representation of
the Young subgroup $\Stab ^0(q)\subset \mathfrak S_n$ on $\mathbb C^n\oplus \mathbb C^n\oplus \mathbb C$,
the direct sum of two copies of a permutation representation and a trivial representation.
Let $U_q$ is a sufficiently small $\Stab (q)$-invariant open
neighborhood of $q\in \widetilde Z^{(n)}$.
If $\Stab ^0(q)$ is of the form $\mathfrak S_{\mu}$ for a partition
$\mu = (\mu _1,\dots, \mu _k)$, $V_q=U_q/\Stab ^0(q)$ is locally isomorphic to
a product of a neighborhood of a general cycle in $\Sym ^n(\mathbb C^2)$ of the form
\[
\sum _{i=1}^k \mu _i a_i\quad (a_i\in \mathbb C^2),
\]
and a complex line $\mathbb C$.
Therefore, it admits a crepant resolution $\widetilde V_q\subset \Hilb ^n(\mathbb C^2)\times \mathbb C$.
By a theorem of Haiman (\cite{H}, Theorem 5.1),
$\widetilde V_q$ can be regarded as an open subset of
$\mathfrak S_n\mbox{-}\Hilb (\mathbb C^{2n+1})$,
so the quotient group $G_q=\Stab (q)/\Stab ^0(q)$ naturally acts on $\widetilde V_q$
and the quotient $\widetilde V_q/G_q$ gives a partial resolution
\[
\psi ^{(n)}_{2,q} : \widetilde V_q/G_q \to \overline U_q
\]
of the image $\overline U_q$ of $U_q$ in $Z^{(n)}=\widetilde Z^{(n)}$.
Since $\overline U_q\subset Z^{(n)}$ has only canonical singularities,
$\psi ^{(n)}_{2,q}$ is again a crepant birational morphism.
As the partial resolution $\psi ^{(n)}_{2,q}$ naturally glue with
the Hilbert-Chow morpihsm
$\psi^{(n) \circ}: Y^{(n) \circ}\to Z^{(n) \circ}$ at the
image of $q$ in $Z^{(n)}$ for every $q=(q_1,q_2)$ with $q_1\in F^0$,
we get an extended crepant partial resolution
$\psi ^{(n)}:Y^{(n)}\to Z^{(n)}$.

\subsection{}
To finish the proof of Theorem \ref{minimal model},
we check that $\psi ^{(n)}$ constructed above satisfies the conditions (i) and (iii).

Let $D_1,\dots ,D_k\subset X(\Delta ^{(n)})$ be torus invariant prime divisors.
At a point $q\in F^0$, they are defined by $y_j=0$ (not all but for some $j$'s)
in the notation of \S\ref{s fixed open}.
Lemma \ref{local eigenvalues} implies that the intersection $D_1\cap \dots \cap D_k$
is (if not empty) transversal to the action of $\mathfrak S_n$.
Therefore, the strict transform $D'_i$ of the image of $D_i$ in $\widetilde V_q$
is smooth divisor and intersecting transversally along the exceptional divisor.
As the quotient group $G_q=\Stab (q)/\Stab ^0(q)$ acts on the coordinate
$(y_1,\dots, y_r)$ via a torus $(\mathbb C^*)^r$,
the normal crossing divisor $\sum D'_i$ is preserved by the action of $G_q$.
Therefore, the singular fiber of $Y^{(n)}\to B$ is a divisor with
$V$-normal crossings. This proves the condition (iii).

The condition (i) is a consequences of the characterization of terminal quotient singularity
(see, \cite{MS} Theorem 2.3).
Let $\tilde \Gamma$ be a unique smooth irreducible divisor on $\widetilde Z^{(n)}$ dominating
$F^{(0^{n-1})}_{(1\; 2)}\subset X(\Delta ^{(n)})$ and $\Gamma$ the image of $\tilde \Gamma$
in $Z^{(n)}$. Lemma \ref{value of age} implies that the age function always satisfies
$a(s;\theta, \bar p)\geqslant n-r$ if $s$ has the partition type of length $r$
(regardless of a choice of primitive root of unity).
In particular, if $a(s;\theta, \bar p)=1$, we necessarily has $r=n-1$,
therefore the associated partition should be $\lambda =(2,1^{n-1})$. Moreover, in that case,
one need to have $\theta _1=\dots =\theta _r$ and $\theta _r=0$, namely $\theta =(0^{n-1})$.
This implies that an exceptional divisor of discrepancy $0$ over $Z^{(n)}$ necessarily dominates $\Gamma$
(see \cite{Rei} Remark (3.2)).
On the other hand, as $\psi ^{(n)}$ is a crepant resolution of
the singularity of $Z^{(n)}$ at the generic point
of $\Gamma$, $Y^{(n)}$ has no crepant exceptional divisor over it. This implies (i)
and completes the proof of Theorem \ref{minimal model}.

\subsection{}
The recent construction of degeneration of Hilbert schemes by Gulbrandsen, Halle, and Hulek
\cite{GHH} seems to be strongly related to the problem.
In this paragraph, we use the notation of \cite{GHH} freely.
We can show that, for the expanded degeneration $X[n]\to \mathbb A^{n+1}$, there is a natural isomorphism
between the GIT quotient of
the stable locus of the relative symmetric product $\Sym ^n(X[n]/\mathbb A^{n+1})^s$
by $G[n]=(\mathbb C^*)^n$ and our $Z^{(n)}$:
\[
\varepsilon ^{(n)}:\Sym ^n(X[n]/\mathbb A^{n+1})^s \git G[n] \overset{\sim}{\lto} Z^{(n)}.
\]
Therefore, the relative Hilbert--Chow morphism
\[
\Hilb ^n(X[n]/\mathbb A^{n+1})^s\to \Sym ^n(X[n]/\mathbb A^{n+1})^s
\]
gives a birational morphism
\[
\psi ^{(n),GHH}: I^n(X/\mathbb A^1)=\Hilb ^n(X[n]/\mathbb A^{n+1})^s \git G[n] \to Z^{(n)}
\]
over the base $B=\mathbb A^1$. As the authors claim in \cite{GHH} that
$I^n(X/\mathbb A^1)$ has only (ablian) quotient singularities and
has trivial canonical bundle, it will immediately follow that $\psi ^{(n),GHH}$ is an extension of
$\psi ^{(n)\circ}$ satisfying the conditions (i)--(iii) of Theorem \ref{minimal model}.
We will discuss the construction of $\varepsilon ^{(n)}$ and the comparison of
$\psi ^{(n),GHH}$ and our $\psi ^{(n)}$ in a forthcoming article \cite{Nf}.

\subsection{}
The theorem asserts that $Y^{(n)}$ is a \emph{relatively minimal model} of the symmetric product $X^{(n)}$ over $B$,
as $K_{Y^{(n)}}$ is numerically trivial over $X^{(n)}$ and $K_{X^{(n)}}\equiv 0$.
The general theory of MMP also suggests that there may be other minimal models of the symmetric product.
Actually, if $n=2$ or $3$,
we can prove that the relative Hilbert scheme $\Hilb ^n (S_2/B)$ is irreducible and
admits a small resolution
\[
H^{(n)}\to \Hilb ^n (S_2/B)\to X^{(n)}
\]
such that the natural map $H^{(n)}\to B$ is \emph{semistable} (see \cite{N}, Theorem 4.3, for the case $n=2$).
Moreover, for $n=2$, we can explicitly write down the flop
$\xymatrix{H^{(2)}\ar@{..>}[r] &Y^{(2)}}$ as follows.

The singular fiber of $p_2:S_2\to B$ consists of two components $S_{2,1}=(x_1=0)$ and $S_{2,2}=(x_2=0)$.
Let $C= S_{2,1}\cap S_{2,2}\cong \mathbb A^1$ be the double line.
The fiber of the relative Hilbert scheme $\Hilb ^2(S_2/B)\to B$ consists of three components:
\[
\Hilb ^2(S_{2,1}),\; \Hilb ^2(S_{2,2}),\; \mbox{and a component } \overline H_{12} \mbox{ birational to }S_0\times S_1.
\]
Non-trivial fiber occurs over a cycle $\gamma =2p\in X^{(2)}$. If the support of $\gamma$ lies in the
smooth locus of $p_2$, the fiber of $\Hilb ^2(S_2/B)\to X^{(2)}$ is $\mathbb P^1$.
Let us denote the associated cycle by $l_1$.
If the support is in the double curve $C$, the fiber is $\mathbb P^2$ and
we denote a class of line in this $\mathbb P^2$ by $l_2$.

The small resolution $h: H^{(2)}\to \Hilb ^2(S_2/B)$ is given by a blowing-up along the (non-Cariter) divisor
$\Hilb ^2(S_{2,1})$.

\begin{claim}
The singular fiber of $H^{(2)}\to B$ consists of
\[
H_{ii} = Bl _{\Hilb ^2(C)}(\Hilb ^2(S_{2,i}))\;(i=1,2),\mbox{\quad and \quad} H_{12}= Bl _{\Delta _C}(S_{2,1}\times S_{2,2}),
\]
where $\Delta _C$ is the diagonal of $C$ in $C\times C\subset S_{2,1}\times S_{2,2}$.
\end{claim}

\begin{proof}
Let $D$ be the diagonal of $S_2\times S_2$ and $W$ the strict transform of $(p_2\times p_2)^{-1}\Delta _B$
in $Bl_{D}(S_2\times S_2)$, where $\Delta _B\cong B$ is the diagonal of $B\times B$. Then, $\Hilb ^2(S_2/B)$ is nothing
but the quotient $W/\mathfrak S_2$. The fiber of $W\to \Delta _B\cong B$ over the origin $0\in B$ consists of
four components
\[
W_{ii} = Bl_{\Delta _{S_{2,i}}} (S_{2,i}\times S_{2,i}), \quad
W_{ij} = Bl_{\Delta _C} (S_{2,i}\times S_{2,j})
\]
with $i,j\in \{1,2\},\, i\neq j$. If we take $\widetilde W=Bl _{W_{11}} W=Bl _{W_{22}} W$, we have an isomorphism
$H^{(2)}\cong \widetilde W/\mathfrak S_2$.
Let us denote by $\widetilde W_{ii},\widetilde W_{ij}$ the strict transforms of $W_{ii}, W_{ij}$, respectively.
It immediately follows that the fiber of $H^{(2)}\to B$ over $0\in B$ consists
of
\[
\begin{aligned}
H_{ii}&\cong \widetilde W_{ii}/\mathfrak S_2=Bl _{\Hilb ^2(C)}(\Hilb ^2(S_{2,i}))
\mbox{\quad and}\\
H_{12}&\cong \widetilde W_{12}\cong W_{12}=Bl_{\Delta _C} (S_{2,1}\times S_{2,2}),
\end{aligned}
\]
noting that $W_{12}$ is smooth and $W_{12}\cap W_{11}$ is a Cartier divisor on $W_{12}$.
\end{proof}

The exceptional divisor $E$ of $Bl_{\Delta _C}(S_{2,1}\times S_{2,2})$ is isomorphic to $\mathbb P^2\times \Delta _C$
and the class of line on the fiber $\mathbb P^2$ is $l_2$.
Here, we remark that the morphism $h$ restricted to the strict transform $C\times C\subset H_{12}$ of
$C\times C\subset S_{2,1}\times S_{2,2}$ is the canonical morphism
\[
C\times C\to \Sym ^2(C),
\]
while $h_{|H_{12}}$ is birational.
\renewcommand{\thefootnote}{\ensuremath{\fnsymbol{footnote}}}
It in particular implies that the component $\overline {H}_{12}$ is \emph{non-normal}.%
\footnote{We remark that the description of $\overline H_{12}$ in \cite{N}, p.419 (appeard as `$Y_{i,i+1}$')
is erroneous.}
One also sees from the description given in the proof of the claim above
that $H_{11}\cap H_{22}$ is $\mathbb P^1$-bundle over $\Hilb ^2(C)=\Sym ^2(C)$:
\[
H_{11}\cap H_{22} = \mathbb P(N_{\Hilb ^2(C)/\Hilb ^2(S_{2,1})}),
\]
and $H_{11}\cap H_{22}\cap H_{12}$ is isomorphic to $C\times C$ that is embedded
in $H_{11}\cap H_{22}$ as a double section over $\Sym ^2(C)$. All the exceptional fibers of
$h$ are isomorphic to $\mathbb P^1$, whose numerical class we denote by $l_3$.
The fiber $\mathbb P^1$ of $H_{11}\cap H_{22}\to \Sym ^2(C)$ is numerically equivalent to $l_3$.

Now, the relative cone of curves $NE(H^{(2)}/X^{(2)})$ is spanned by $l_1,\, l_2,$ and $l_3$.
An easy calculation shows that
\[
H_{12}\cdot l_1=0,\quad H_{12}\cdot l_2=-2,\quad \mbox{and\quad} H_{12}\cdot l_3=2.
\]
As the canonical bundle of $H^{(2)}$ is trivial by \cite{N}, Theorem 4.3,
$(H^{(n)},\varepsilon H_{12})$ is klt for a sufficiently small positive rational number $\varepsilon$, and
Cone Theorem guarantees that
there is an extremal contraction of $l_2$, which is a small contraction that contracts $E$.
One sees that $Y^{(2)}$ is nothing but its flop. Actually, the flop produces family of $\mathbb P^1$ over
$\Delta _C$ passing through a $\frac 12(1,1,1,1)$-singularity coming from the fixed point locus
with the angle type $\theta =(1/2)$. This is a locally trivial family of toric flop that is called ``Francia flop'' in
\cite{Ka} (Example 5.1 and Definition 4.1).

%%%%%%%%%%%%%%%%%%%%%%%%%%%

\end{document}